\documentclass[reqno,11pt]{amsart}

\usepackage[top=2.0cm,bottom=2.0cm,left=3cm,right=3cm]{geometry}
\usepackage{amsthm,amsmath,amssymb,dsfont}
\usepackage{mathrsfs,amsfonts,functan,extarrows,mathtools}
\usepackage[colorlinks]{hyperref}

\usepackage{marginnote}
\usepackage{xcolor}
\usepackage{stmaryrd}
\usepackage{esint}
\usepackage{graphicx}
\newtheorem{theorem}{Theorem}[section]
\newtheorem{proposition}{Proposition}[section]

\newtheorem{remark}{Remark}[section]
\newtheorem{lemma}{Lemma}[section]

\numberwithin{equation}{section}
\allowdisplaybreaks

\def\p{\partial}

\def\d{\mathrm{d}}
\def\dr{\mathbf{d}}
\def\no{\nonumber}
\def\eps{\varepsilon}
\def\div{\mathrm{div}}
\def\u{\mathbf{u}}
\def\A{\mathbf{A}}
\def\B{\mathbf{B}}

\def\D{\mathbf{D}}
\def\S{\mathcal{S}}
\def\R{\mathbb{R}}
\def\dre{\mathbf{d}^\epsilon_{R}}
\def\w{\mathbf{w}}

\def\be{\begin{equation}}
\def\ee{\end{equation}}
\def\bea#1\eea{\begin{align}#1\end{align}}

\newcounter{wronumber}

\begin{document}
\title[Wave map limit]
			{A scaling limit from the wave map\\ to  the heat flow into $\mathbb{S}^2$}

\author[N. Jiang]{Ning Jiang}
\address[Ning Jiang]{\newline School of Mathematics and Statistics, Wuhan University, Wuhan, 430072, P. R. China}
\email{njiang@whu.edu.cn}

\author[Y.-L. Luo]{Yi-Long Luo}
\address[Yi-Long Luo]
		{\newline School of Mathematics and Statistics, Wuhan University, Wuhan, 430072, P. R. China}
\email{yl-luo@whu.edu.cn}

\author[S. J. Tang]{Shaojun Tang}
\address[Shaojun Tang]
       {\newline School of Mathematics and Statistics, Wuhan University, Wuhan, 430072, P. R. China}
       \email{shaojun.tang@whu.edu.cn}

\author[A. Zarnescu]{Arghir Zarnescu}
\address[Arghir Zarnescu]
        {\newline BCAM, Basque Center for Applied Mathematics, Mazarredo 14, 48009 Bilbao, Spain}
        \email{azarnescu@bcamath.org}

\thanks{ \today}

\maketitle

\begin{abstract}
 In this paper we study a limit connecting a scaled wave map with the heat flow into the unit sphere $\mathbb{S}^2$. We show quantitatively how  that the two equations are connected by means of  an initial layer correction.
 This limit is motivated as a first step into  understanding the limit of zero inertia for the hyperbolic-parabolic Ericksen-Leslie's liquid crystal model. \end{abstract}

%\vspace*{10pt}

%\phantomsection
%\addcontentsline{toc}{section}{\contentsname}

%\tableofcontents

%%%%%%%%%%%%%%%%%%%%%%%%%%%%%%%%%%%%%%%%%%%%%%%%%%%%%%%%%%%%%%%%%%%%%%%%%%%
%%%%%%%%%%%%%%%%%%%%%%%%%%%%%%%%%%%%%%%%%%%%%%%%%%%%%%%%%%%%%%%%%%%%%%%%%%%

\section{Introduction}

\subsection{Wave map and heat flow}
We consider a hyperbolic system for functions $\dr:\mathbb{R}^+\times\mathbb{R}^3\to \mathbb{S}^2$:
\begin{equation}\label{WM-0}
  \p_t \dr = -\square \dr + (|\nabla \dr|^2-|\p_t \dr|^2) \dr\,,
\end{equation} subject to  initial data: for any $x\in \mathbb{R}^3$,
\begin{equation}
  \dr|_{t=0} = \dr^{0}(x) \in \mathbb{S}^2 \, , \ \partial_t \dr|_{t=0} = \tilde{\dr}^{0} (x) \in \mathbb{R}^3\,,\dr^{0}(x)\cdot \tilde{\dr}^{0}(x)= 0\,,
\end{equation}
where $\square = \p_{tt}- \Delta$ is the standard wave operator, and the compatibility condition $\dr^{0}\cdot \tilde{\dr}^{0}= 0$ on the initial data is due to the fact that $|\dr^0|\equiv 1$.

\bigskip
 The system \eqref{WM-0} is a wave map from $\mathbb{R}^3$ to the unit sphere $\mathbb{S}^2$, with a damping term $\p_t \dr$. One way of interpreting this system is as follows: setting the righthand side of \eqref{WM-0} equal to 0, we obtain $\square \dr = (|\nabla \dr|^2-|\p_t \dr|^2) \dr$. This  is the well-known {\it wave map}, which can be characterized variationally as a critical point of the functional
\begin{equation}\label{WM-functional}
  \mathcal{A}(\dr) = \frac{1}{2}\iint (|\nabla \dr|^2-|\p_t \dr|^2)\,\d x\d t\,,
\end{equation}
among maps $\dr$ satisfying the target constraint, $\dr: \mathbb{R}^+\times \mathbb{R}^3 \rightarrow \mathbb{S}^2$. Thus the full system \eqref{WM-0} can be viewed as a ``gradient flow" of the functional \eqref{WM-functional}.

\bigskip Another gradient flow can be obtained by formally dropping some terms out of the previous system, and obtaining the {\it heat flow}

\begin{equation}\label{Heat-F}
  \p_t \dr = \Delta \dr + |\nabla \dr|^2 \dr\,.
\end{equation}

Similarly as before setting the right-hand side equal to zero we obtain the equations for  the harmonic map from $\mathbb{R}^3$ to the unit sphere $\mathbb{S}^2$ namely
\begin{equation}\label{HM}
  \Delta \dr + |\nabla \dr|^2 \dr =0\,,
\end{equation}
which is a critical point of the energy functional
\begin{equation}\label{HM-functional}
  E(\dr) = \frac{1}{2} \int |\nabla \dr|^2\, \d x\,.
\end{equation}

%We emphasize that here for simplicity we only consider the domain is the Eulerian space $\mathbb{R}^3$ and the target is the unit sphere $\mathbb{S}^2$, the general harmonic maps can be defined between two Riemannian manifolds. It is well-known in the pioneering work of Eells-Sampson, the ``gradient flow" of the energy functional \eqref{HM-functional}, i.e. the following heat flow was started to be studied:
%\begin{equation}\label{Heat-F}
%  \p_t \dr = \Delta \dr + |\nabla \dr|^2 \dr\,.
%\end{equation}

There exist deep relations between the two systems, \eqref{WM-0} and \eqref{Heat-F} and one way to see this is by
considering the following parabolic scaling:
\begin{equation}
  \dr^\eps(t,x) := \dr(\tfrac{t}{\eps}, \tfrac{x}{\sqrt{\eps}}),
\end{equation}

Then, $\dr^\eps$ satisfies the following scaled wave map:
\begin{equation}\label{Wave-Map}
        \partial_t \dr^\eps = -(\eps \partial_{tt} -\Delta) \dr^\eps + (  |\nabla \dr^\eps|^2- \eps |\partial_t \dr^\eps|^2 ) \dr^\eps \,,
\end{equation}
on $\R^+ \times \R^3$.
For this scaled system we take  the initial values independent of $\varepsilon$, namely:
\begin{equation}\label{IC}
  \dr^\eps \big{|}_{t=0} = \dr^{in}(x) \in \mathbb{S}^2 \, , \ \partial_t \dr^\eps \big{|}_{t=0} =\tilde\dr^{in}(x) \in \mathbb{R}^3\,.
\end{equation}

The finite-time behaviour of the limit $\eps\to 0$ for the system \eqref{Wave-Map} with initial data \eqref{IC} is the focus of this paper.   It is easy to see that letting $\eps = 0$ in \eqref{Wave-Map} will formally give the heat flow \eqref{Heat-F}. However, a refined analysis and the introduction of an {\it initial layer}   is needed in order to overcome the singular character of this limit and understand the relationship between the system \eqref{Wave-Map} and its formal limit, as it will be seen in the Theorem~\ref{Main-Thm} below.

\medskip
{\bf Notations and conventions:} Throughout this paper, we use the following standard notations:
$$ |\dr|^p_{L^p} = \int_{\R^3} |\dr|^p \d x \, , \ |\dr|_{H^k} = \sum_{\gamma \leq k} |\nabla^\gamma \dr|_{L^2} \, , \ |\dr|_{\dot{H}^k} = \sum_{1 \leq \gamma \leq k} |\nabla^\gamma \dr|_{L^2} \, .$$
Additionally, for the Hilbert space $L^2 \equiv L^2(\d x, \R^3)$, we use the following notation to denote the standard inner product:
$$ \left\langle f, g \right\rangle = \int_{\R^3} f g \d x \, .$$
Furthermore, if there is a generic constant $C>0$ such that the inequality $f(t) \leq C g(t)$ holds for all $t \geq 0$, we denote this inequality by
$$ f \lesssim g \, . $$

\subsection{Ericksen-Leslie's hyperbolic liquid crystal model}

Our motivation for considering the previously mentioned  limit comes from the hydrodynamic theory of nematic liquid crystals.

The most widely accepted equations of nematics were proposed   by Ericksen \cite{Ericksen-1961-TSR, Ericksen-1987-RM, Ericksen-1990-ARMA} and Leslie \cite{Leslie-1968-ARMA, Leslie-1979} in the 1960's (see for more details Section 5.1 of \cite{Lin-Liu-2001} ). The general hyperbolic-parabolic Ericksen-Leslie system consists of an equation for velocity $u$ of the centers of mass of the rod-like molecules, coupled with an equation for the direction $d$ of these molecules. More specifically we have the following equations (in non-dimensional form):
\begin{equation}\label{ParaHyper-LC-general}
  \begin{aligned}
    \left\{ \begin{array}{c}
      \partial_t \u + \u \cdot \nabla \u - \frac{1}{2} \mu_4 \Delta \u + \nabla p = - \div (\nabla \dr \odot \nabla \dr) + \div \tilde{\sigma}\, , \\
      \div \u = 0\, ,\\
     \eps D_\u^2 \dr = \Delta \dr + \gamma \dr + \lambda_1 (D_\u \dr- \B \dr) + \lambda_2 \A \dr
    \end{array}\right.
  \end{aligned}
\end{equation}
on $\R^+ \times \R^3$ with constraint $|\dr|=1$, where $\A = \frac{1}{2} (\nabla \u + \nabla \u^\top)$ and $\B = \frac{1}{2} (\nabla \u - \nabla \u^\top) $, $D_\u f = \partial_t f + \u \cdot \nabla f$ is the material derivative of $f$ respect to the vector $\u$, $D_\u^2 \dr = \partial_t D_\u \dr + \u \cdot \nabla D_\u \dr$. The  Lagrangian $\gamma$ that enforces the unit-length constraint $|\dr|=1$ is given by
$$\gamma \equiv \gamma (\u, \dr, D_\u \dr) = - \eps |D_\u \dr|^2 + |\nabla \dr|^2 - \lambda_2 \dr^\top \A \dr \, . $$
The stress tensor appearing in the equation for $u$  is given by:
\begin{align*}
   \no \tilde{\sigma}_{ij} \equiv \big{(}\tilde{\sigma}(\u, \dr, D_\u \dr)\big{)}_{ij}  = & \mu_1 \dr_k \dr_p \A_{kp} \dr_i \dr_j + \mu_2 \dr_j ((D_\u \dr)_i + \B_{ki} \dr_k)  \\
   &+ \mu_3 \dr_i ((D_\u \dr)_j + \B_{kj} \dr_k)    +  \mu_5 \dr_j \dr_k \A_{ki} + \mu_6 \dr_i \dr_k \A_{kj} \, .
\end{align*}

 The constant $\eps > 0$ measures the inertial effects. The constants $\mu_i$ $(1 \leq i \leq 6)$ are known as Leslie coefficients and one has  $\mu_4 > 0$. Furthermore, we have:
$$ \lambda_1 = \mu_2 - \mu_3\, , \lambda_2 = \mu_5 - \mu_6 \, , \mu_2 + \mu_3 = \mu_6 - \mu_5\, , $$
where the last relation is called {\em Parodi } relation. For the more background and derivation of \eqref{ParaHyper-LC-general}, see \cite{Leslie-1968-ARMA} and \cite{Jiang-Luo-2017}.

For any fixed $\eps > 0$, in \cite{Jiang-Luo-2017} the first two named authors of the current paper proved the local well-posedness of the system \eqref{ParaHyper-LC-general} under assumptions on the Leslie coefficients which ensure the dissipativity of the basic energy law, and global well-posedness with small initial data under further damping effect, i.e. $\lambda_1 <0$.

 As noted in the ``Conclusion" section of \cite{Jiang-Luo-2017},  the inertial constant $\eps>0$ is, physically,    in most common non-dimensionalisations and materials,  very small. Formally, letting $\eps=0$ will give the parabolic Ericksen-Leslie system which is basically a coupling of Navier-Stokes equations and an extension of the heat flow to the unit sphere. However it is a very challenging task to  obtaining  estimates uniform in $\eps$ for the full system \eqref{ParaHyper-LC-general}, in order to understand the limit  $\eps\rightarrow 0$. In the current paper, the problem we consider what appears as a  simple instance of this general problem, namely the case where the bulk velocity $\u= \mathbf{0}$ and the coefficient $\lambda_1=-1$ in \eqref{ParaHyper-LC-general}. For this case, the system \eqref{ParaHyper-LC-general} is reduced to the scaled wave map \eqref{Wave-Map}, i.e. the wave map \eqref{Wave-Map} with a damping can be regarded as an Ericksen-Leslie's  liquid crystal flow unaffected by the fluid velocity.

\subsection{Initial layer and the main result}
As mentioned in the previous  two subsections the formal limit of the equation \eqref{Wave-Map}, obtained by setting $\eps=0$ is provided by the heat flow for functions  with values into  $\mathbb{S}^2$:

\begin{equation}\label{HF}
      \partial_t \dr_0 = \Delta \dr_0 + |\nabla \dr_0|^2 \dr_0 \, , \dr_0 \in \mathbb{S}^2 \,,
\end{equation}

The limit we consider is a singular limit, as the character of the equations changes, from a hyperbolic-type system for $\eps>0$ to a parabolic system for $\eps=0$. An immediate manifestation of the difference between the two types of equations is related to the initial conditions, which for the limit equation take the form:

\begin{equation}\label{IC-HF}
  \dr_0 \big{|}_{t=0} = \dr^{in}(x) \in \mathbb{S}^2\,.
\end{equation}

Thus, we note that  the wave map is a system of hyperbolic equations with two initial conditions, while the heat flow is a parabolic system with only one initial condition. Usually the solution of the heat flow does not satisfy the second initial condition in \eqref{IC}. This disparity between the initial conditions of the wave map \eqref{Wave-Map} and of the heat flow \eqref{HF} indicates that in one should expect an  ``initial layer" in time, appearing in the limiting process $\eps\rightarrow 0$. A formal derivation (postponed for later, in Section \ref{Sct-Formal-Anly}) indicates that this should be of the form:

\begin{equation}\label{IL-Construction}
  \begin{aligned}
    \dr_0^I \big{(} \tfrac{t}{\eps} , x \big{)} = & - \eps ( \tilde{\dr}^{in} (x) - \partial_t \dr_0 (0,x) ) \exp ( - \tfrac{t}{\eps} )\\
    = & - \eps ( \tilde{\dr}^{in} (x) - \Delta \dr^{in}(x) - |\nabla \dr^{in}(x)|^2 \dr^{in}(x) ) \exp ( - \tfrac{t}{\eps} ) \\
    \equiv & - \eps \D (x) \exp ( - \tfrac{t}{\eps} ) \, ,
  \end{aligned}
\end{equation}
where $\D(x)$ is defined as
$$ \D(x) \equiv \tilde{\dr}^{in} (x) - \Delta \dr^{in}(x) - |\nabla \dr^{in}(x)|^2 \dr^{in}(x) \,.$$

Our study of the limit from the wave map \eqref{Wave-Map} to the heat flow \eqref{HM} is inspired by the classical approach of Caflisch on the compressible Euler limit of the Boltzmann equation \cite{Caflisch}. This approach is based on the Hilbert expansion in which the leading term is given by solutions of the limit equation. The Caflisch's approach assumes that a solution of the limiting equation (which in our case  is the heat flow \eqref{HM}) is known beforehand. Then the solution to the original equation (which in our case is the wave map \eqref{Wave-Map}) can be constructed around the limiting equation with perturbations as  expansions in powers of $\eps$. Based on the arguments above and the formal analysis in Section~\ref{Sct-Formal-Anly}, in the expansions, besides the heat flow, the leading term should also include an initial layer. More specifically, we take the following ansatz of  the solution $\dr^\eps$ to the system \eqref{Wave-Map}:
\begin{equation}\label{Ansatz-3}
  \dr^\eps(t,x) = \dr_0 (t,x) + \dr_0^I \big{(} \tfrac{t}{\eps} , x \big{)} + \sqrt{\eps} \dr_R^\eps(t,x) \, ,
\end{equation}
where $\dr_0(t,x)$ obeys the heat flow \eqref{HF} and the initial layer $\dr_0^I \big{(} \tfrac{t}{\eps} , x \big{)} $ is defined in \eqref{IL-Construction}. Plugging \eqref{Ansatz-3} into the system \eqref{Wave-Map}, the remainder term $\dr_R^\eps (t,x)$ must satisfy the system
\begin{equation}\label{Remainder-Eq}
  \partial_{tt} \dr_R^\eps + \tfrac{1}{\eps} \partial_t \dr_R^\eps - \tfrac{1}{\eps} \Delta \dr_R^\eps = \mathcal{S} (\dr_R^\eps) + \mathcal{R} (\dr_R^\eps)
\end{equation}
with the initial conditions
\begin{equation}\label{IC-Remainder}
  \dr_R^\eps(0,x) = \sqrt{\eps} \D(x) \, , \ \partial_t \dr_R^\eps (0,x) = 0 \, ,
\end{equation}
where the singular term $\mathcal{S} (\dr_R^\eps)$ is
\begin{equation*}
  \begin{aligned}
    \mathcal{S} (\dr_R^\eps) = & - \tfrac{1}{\sqrt{\eps}} \partial_{tt} \dr_0 - \tfrac{1}{\sqrt{\eps}} \Delta \D (x) \exp (- \tfrac{t}{\eps}) + \tfrac{1}{\eps} |\nabla \dr_0|^2 \dr_R^\eps + \tfrac{1}{\sqrt{\eps}} |\nabla \dr_R^\eps|^2 \dr_0 \\
    &  - \tfrac{1}{\sqrt{\eps}} | \partial_t \dr_0 + \D (x) \exp(-\tfrac{t}{\eps}) |^2 \dr_0   - \tfrac{1}{\sqrt{\eps}} |\nabla \dr_0|^2 \D (x) \exp(- \tfrac{t}{\eps}) \\
    & + \tfrac{2}{\sqrt{\eps}} ( \nabla \dr_0 \cdot \nabla \dr_R^\eps ) \dr_R^\eps - \tfrac{2}{\sqrt{\eps}} ( \nabla \dr_0 \cdot \nabla \D (x) ) \exp( - \tfrac{t}{\eps} ) \dr_0 + \tfrac{2}{\eps} (\nabla \dr_0 \cdot \nabla \dr_R^\eps) \dr_0 \, ,
  \end{aligned}
\end{equation*}
and the regular term $ \mathcal{R} (\dr_R^\eps) $ is
\begin{equation*}
  \begin{aligned}
    \mathcal{R} (\dr_R^\eps) =& - \big{|} \partial_t \dr_0 + \D(x) \exp( - \tfrac{t}{\eps} ) + \sqrt{\eps} \partial_t \dr_R^\eps \big{|}^2 \big{[} - \sqrt{\eps} \D (x) \exp( - \tfrac{t}{\eps} ) + \dr_R^\eps \big{]} \\
    & - \big{[} 2 ( \partial_t \dr_0 + \D(x) \exp( - \tfrac{t}{\eps} ) ) \cdot \p_t \dr_R^\eps + \sqrt{\eps} |\partial_t \dr_R^\eps|^2 \big{]} \dr_0 - 2 ( \nabla \D(x) \cdot \nabla \dr_R^\eps ) \exp( - \tfrac{t}{\eps} ) \dr_0 \\
    & + \sqrt{\eps} |\nabla \D(x)|^2 \exp( - \tfrac{2t}{\eps} ) \dr_0 + |\nabla \dr_R^\eps|^2 \dr_R^\eps - 2 (\nabla \dr_0 \cdot \nabla \D(x)) \exp( - \tfrac{t}{\eps} ) \dr_R^\eps \\
    & - 2 ( \nabla \dr_0 \cdot \nabla \dr_R^\eps ) \D(x) \exp( - \tfrac{t}{\eps} ) - 2 \sqrt{\eps} (\nabla \D(x) \cdot \nabla \dr_R^\eps) \exp( - \tfrac{t}{\eps} ) \dr_R^\eps \\
    & - \sqrt{\eps} |\nabla \dr_R^\eps|^2 \D(x) \exp( - \tfrac{t}{\eps} ) + \eps |\nabla \D(x)|^2 \exp( - \tfrac{2t}{\eps} ) \dr_R^\eps - \eps^\frac{3}{2} |\nabla \D(x)|^2 \exp( - \tfrac{3t}{\eps} ) \D(x) \\
    & + 2 \sqrt{\eps} (\nabla \dr_0 \cdot \nabla \D(x) ) \exp( - \tfrac{2t}{\eps} ) \D(x) + 2 \eps ( \nabla \D(x) \cdot \nabla \dr_R^\eps ) \exp( - \tfrac{2t}{\eps} ) \D(x) \, .
  \end{aligned}
\end{equation*}

According to Eells-Sampson's classical result in \cite{Eells-Sampson-AJM1964}, for the heat flow \eqref{HF} on the unit sphere $\mathbb{S}^2$, one can have the following results of local well-posedness:
\begin{proposition}\label{Prop-HF}
  For any given $\dr^{in} \in \mathbb{S}^2$ satisfying $ \dr^{in} \in \dot{H}^k (\R^3)$ for any integer $k > 2$, there exists a time $T = T(|\dr^{in}|_{\dot{H}^k}) > 0$ such that \eqref{HF} admits a unique classical solution $\dr_0 \in L^\infty(0,T; \dot{H}^k) \cap L^2 (0,T;\dot{H}^{k+1})$. Moreover, there is a constant $C^* = C^*(|\dr^{in}|_{\dot{H}^k} , T ) > 0$ such that the solution $\dr_0$ satisfies
  \begin{equation*}
    |\dr_0|^2_{L^\infty(0,T; \dot{H}^k)} + |\nabla \dr_0|^2_{L^2(0,T; H^k)} \leq C^*\, .
  \end{equation*}
\end{proposition}

The proof can be found in Chapter 5 in the book \cite{Lin-Wang-2008}.

Now we state the main result of this paper:

\begin{theorem}\label{Main-Thm}
We consider vector fields $\dr^{in}:\mathbb{R}^3\to \mathbb{S}^2$ and $\tilde{\dr}^{in}:\mathbb{R}^3\to \mathbb{R}^3$, satifying the compatibility condition $\dr^{in}\cdot\tilde{\dr}^{in}\equiv 0$. Assume that $\nabla \dr^{in} \in H^6$, $\tilde{\dr}^{in} \in H^5$, and let $T> 0$ be the time interval of existence of the solution of the heat flow \eqref{HF} with initial condition $\dr^{in}$, determined in Proposition \ref{Prop-HF}.

Then, there exists an $\eps_0 \equiv \eps_0 (|\nabla \dr^{in} |_{H^6}, |\tilde{\dr}^{in}|_{H^5}, T) \in (0,\frac 12)$ such that for all $\eps \in (0, \eps_0)$ we have that on the interval $[0,T]$ the wave map equation  \eqref{Wave-Map} with the initial conditions \eqref{IC} admits a unique solution with the form \eqref{Ansatz-3}, i.e.,
  \begin{equation*}
    \dr^\eps(t,x) = \dr_0 (t,x) + \dr_0^I \big{(} \tfrac{t}{\eps} , x \big{)} + \sqrt{\eps} \dr_R^\eps(t,x) \, ,
  \end{equation*}
  where  $\dr_0$ is the solution of the heat flow \eqref{HF} with initial condition $\dr^{in}$ and $\dr_0^I ( \frac{t}{\eps}, x)$ is the initial layer \eqref{IL-Construction}. Moreover, there exists a positive constant $C_0 = C_0 (\dr^{in}, \tilde{\dr}^{in}, T) > 0$, such that the remainder term $\dr_R^\eps$ satisfies the bound
  \begin{equation}\label{Unif-Bounds-Remd}
     |\partial_t \dr_R^\eps|^2_{L^\infty(0,T; H^2) } + \tfrac{1}{\eps} |\dr_R^\eps|^2_{L^\infty( 0, T; H^3) }  \leq C_0
  \end{equation}
for all $\eps \in (0,  \eps_0)$.
\end{theorem}

%\begin{remark}
%  From the uniform bounds \eqref{Unif-Bounds-Remd} of the remainder term $\dr_R^\eps$ and the structure of the initial layer $\dr_0^I \big{(} \tfrac{t}{\eps} , x \big{)}$ shown in \eqref{IL-Construction}, the expansion \eqref{Ansatz-3} can be replaced by
%  \begin{equation}\label{Expn-general}
%    \dr^\eps(t,x) = \dr_0 (t,x) + \dr_0^I \big{(} \tfrac{t}{\eps} , x \big{)} + \eps^\alpha \dr_R^\eps(t,x)
%  \end{equation}
 % for all $\alpha \in ( 0, \frac{1}{2} ]$.
%\end{remark}

\begin{remark}
The rate of convergence we obtain is optimal. Indeed, in order to see this, it suffices to note that the limit we study contains as a particular case the linear {\it scalar} case of the singular limit of the damped wave equation to the heat equation.

Indeed, let us consider a solution of the scalar damped wave  equation:

\be
\epsilon\partial_{tt} \theta^\eps+\partial_t \theta^\eps=\Delta \theta^\eps
\ee for $\theta^\eps:\R\times\R^3\to\R$ with initial datas:

\be
\theta^\eps(0,x)=\theta_0(x),\partial_t \theta^\eps(0,x)=\theta_1(x)
\ee all smooth functions.

Also consider the solution of the heat equation:

\be
\partial_t \theta^0=\Delta \theta^0
\ee for $\theta^0:\R\times\R^3\to\R$ with initial data:

\be
\theta^0(0,x)=\theta_0(x),
\ee all smooth functions.

Denoting $n^0(t,x):=(\cos\theta^0 (t,x),\sin\theta^0 (t,x),0)$ and $n^\eps(t,x):=(\cos\theta^\eps (t,x),\sin\theta^\eps(t,x),0)$ we have that $n^0$ is a solution of the heat-flow:

\be
\partial_t n^0=\Delta n^0+n^0|\nabla n^0|^2
\ee with initial data

\be
n^0(0,x)=n_0(x):=(\cos\theta_0,\sin\theta_0,0)
\ee while $n^\eps$ is a solution of the wave-map flow:

\be
\eps(\partial_{tt}n^\eps+n^\eps|\partial_t n^\eps|^2)+\partial_t n^\eps=\Delta n^\eps+n^\eps|\nabla n^\eps|^2
\ee with initial datas:

\bea
n^\eps(0,x)&=n_0(x):=(\cos\theta_0(x),\sin\theta_0(x),0), \\
\partial_t n^\eps(0,x)&=n_1(x)=(-\sin\theta_0(x),\cos\theta_0(x),0)\theta_1(x)
\eea

Taking $\theta_1=\Delta \theta_0$ the claimed optimality of the rate of converge is shown in \cite{chill-haraux}.
\end{remark}

\bigskip
A rigorous justification of the formal expansion \eqref{Ansatz-3} in the context of classical solutions is provided in this paper. For the original wave map \eqref{Wave-Map} with a damping  the energy bounds of $\dr^\eps$ uniform in small $\eps > 0$ do not seem  available. By taking the expansion \eqref{Ansatz-3} of the solutions $\dr^\eps$ to the system \eqref{Wave-Map} with the initial conditions \eqref{IC}, one can yield a remainder system \eqref{Remainder-Eq}-\eqref{IC-Remainder} of $\dr_R^\eps$. Although the remainder system \eqref{Remainder-Eq} with the initial data conditions \eqref{IC-Remainder} is still nonlinear and singular (with singular terms of the type $\frac{1}{\eps} \cdot$), it has weaker nonlinearities than  the original system \eqref{Wave-Map}. More precisely, after using the expansion \eqref{Ansatz-3}, the nonlinear term $(- \eps |\partial_t \dr^\eps|^2 + |\nabla \dr^\eps|^2) \dr^\eps$ is replaced by either linear terms (the unknown $\dr^\eps$ is superseded by the known $\dr_0$) or a nonlinear term with the same form but with some higher order powers of $\eps$ in front. So, by standard energy estimates, we can get uniform bounds in small $\eps > 0$ of the remainder system \eqref{Remainder-Eq}-\eqref{IC-Remainder}.

\bigskip
The organization of this paper is as follows: in next section, we give the formal analysis for the asymptotic behavior of the wave map \eqref{Wave-Map} with a damping and  initial conditions \eqref{IC} as the inertia density $\eps \rightarrow 0$ by constructing the initial layer $\dr_0^I (\frac{t}{\eps}, x)$ to deal with the compatibility of the original initial conditions \eqref{IC} and the initial condition of the limit system \eqref{HF}. In Section \ref{Sct-Unif-Bounds}, we estimate the uniform energy bounds on small $\eps > 0$ of the remainder system \eqref{Remainder-Eq}-\eqref{IC-Remainder}. Finally, based on the uniform energy estimates in the previous section, Theorem \ref{Main-Thm} of this paper is proved in Section \ref{Sct-Proof}.

\section{Formal Analysis}\label{Sct-Formal-Anly}

In this section we present the formal analysis of the limit  $\eps\to 0$ for  the damped wave map \eqref{Wave-Map} with the initial conditions \eqref{IC}. Out of the equation \eqref{Wave-Map} we note that the formal limit, obtaining by setting $\eps=0$, is the heat flow system \eqref{HF} for functions taking values into $\mathbb{S}^2$. We can then naturally  take the  ansatz
\begin{equation}\label{Ansatz-1}
  \dr^\eps (t,x) = \dr_0 (t,x) + \hat\dr_R^\eps(t,x) \, ,
\end{equation}
where $\dr_0 (t,x)$ is a solution of  the heat flow system \eqref{HF} and $\hat\dr_R^\eps (t,x)$ satisfies a hyperbolic system, formally similar to  \eqref{Wave-Map} but without the geometric constraint of taking values into $\mathbb{S}^2$.

 If the ansatz \eqref{Ansatz-1} were reasonable then $\hat\dr_R^\eps (t,x) = O(\eps^\alpha)$ in some sense for some $\alpha > 0$ as $\eps > 0$ is small enough. However, by the second initial condition in \eqref{IC} and the heat flow system \eqref{HF}, we know that
\begin{equation*}
  \begin{aligned}
    \partial_t \hat\dr_R^\eps (0,x) =& \partial_t \dr^\eps (0,x) - \partial_t \dr_0 (0,x) \\
    =& \tilde{\dr}^{in}(x) - \partial_t \dr_0 (0,x) \, ,
  \end{aligned}
\end{equation*}
which will not go to 0 as $\eps \rightarrow 0$ for arbitrarily given vectors $\tilde{\dr}^{in}(x)$ and $\dr^{in}(x)$. As a consequence, $\hat\dr_R^\eps (t,x) \neq O(\eps^\alpha)$ uniformly in time for any $\alpha > 0$, and then the ansatz \eqref{Ansatz-1} is not satisfactory.

Therefore, in order to compensate the effect of the initial data, we need to introduce a correction term $\dr_0^I \big{(} \frac{t}{\eps^\beta} , x \big{)}$ for some $\beta > 0$ to be determined, called {\em initial layer}, such that the second initial condition in \eqref{IC} is satisfied by $\dr_0(t,x) + \dr_0^I \big{(} \frac{t}{\eps^\beta} , x \big{)} $ without disturbing too much the first initial condition in \eqref{IC}, namely
$$ \dr_0^I \big{(} \tfrac{0}{\eps^\beta} , x \big{)} = O(\eps^\alpha) $$
for some $\alpha > 0$ as $\eps \rightarrow 0$. Thus we take the alternative ansatz

\begin{equation}\label{Ansatz-2}
  \dr^\eps(t,x) = \dr_0 (t,x) + \dr_0^I \big{(} \tfrac{t}{\eps^\beta} , x \big{)} + \sqrt{\eps}\dr_R^\eps(t,x) \, ,
\end{equation} where the power $\sqrt{\eps}$ in front of the remainder term is motivated by the scaling we chose. This measures the rate of convergence, as it will be shown in the proof of Theorem~\ref{Main-Thm}.

Recalling that $\dr_0$ is a solution of the heat flow \eqref{HF}, we plug \eqref{Ansatz-2} into the system \eqref{Wave-Map} and obtain:

\begin{equation*}
  \begin{aligned}
    & \eps \partial_{tt} (\dr_0 +\sqrt{\eps} \dr_R^\eps)  + \sqrt{\eps}\partial_t \dr_R^\eps + \big{(} \eps \partial_{tt} \dr_0^I + \partial_t \dr_0^I \big{)} - \Delta (\dr_0^I +\sqrt{\eps} \dr_R^\eps) \\
    =&  - \eps |\partial_t (\dr_0 + \dr_0^I +\sqrt{\eps} \dr_R^\eps)|^2 (\dr_0 + \dr_0^I +\sqrt{\eps} \dr_R^\eps) \\
    & + |\nabla (\dr_0 + \dr_0^I + \sqrt{\eps}\dr_R^\eps)|^2  (\dr_0^I +\sqrt{\eps} \dr_R^\eps)  + \big{[} 2 \nabla \dr_0 \cdot \nabla ( \dr_0^I +\sqrt{\eps} \dr_R^\eps ) + |\nabla ( \dr_0^I +\sqrt{\eps} \dr_R^\eps)|^2 \big{]} \dr_0 \, .
  \end{aligned}
\end{equation*}

Then we construct the initial layer in order to cancel certain  time-dependent terms in the previous equations and to accommodate the discrepancy in the initial data , namely we take $\dr_0^I$ satisfying the $x$-dependent ODE and the initial-data condition:
\begin{equation*}
  \left\{
    \begin{array}{l}
      \eps \partial_{tt} \dr_0^I \big{(} \tfrac{t}{\eps^\beta} , x \big{)} + \partial_t \dr_0^I \big{(} \tfrac{t}{\eps^\beta} , x \big{)} = 0 \, , \\
      \partial_t \dr_0^I \big{(} \tfrac{0}{\eps^\beta},x \big{)} = \tilde{\dr}^{in} (x) - \partial_t \dr_0(0,x) \, .
    \end{array}
  \right.
\end{equation*}
%We stretch the $t$-coordinate by the transformation $t = \eps^\beta \tau$, then the above second order ODE equations can be written as
%\begin{equation}\label{IL-Eq}
%  \left\{
%    \begin{array}{l}
%     \partial_{\tau \tau} \dr_0^I (\tau, x) + \eps^{\beta-1} \partial_\tau \dr_0^I (\tau, x) = 0 \, , \\
%      \partial_\tau \dr_0^I (0 , x) = \eps^\beta \big{(} \tilde{\dr}^{in} (x) - \partial_t \dr_0 (0,x) \big{)} \, .
%    \end{array}
%  \right.
%\end{equation}
Furthermore, since $\dr_0^I$ is an initial layer, the following condition at infinity is required:
\begin{equation}\label{IL-BC}
  \dr_0^I (\infty, x) = \lim\limits_{y\rightarrow \infty} \dr_0^I (y, x) = 0 \, .
\end{equation}
By solving the ODE system with the given boundary conditions we have
%$$ \dr_0^I (\tau, x) = - \eps \big{(} \tilde{\dr}^{in} (x) - \partial_t \dr_0 (0,x) \big{)} \exp \big{(} - \eps^{\beta - 1} \tau \big{)} \, , $$
% namely,
\begin{equation}\label{IL}
  \dr_0^I \big{(} \tfrac{t}{\eps^\beta} , x \big{)} = - \eps \big{(} \tilde{\dr}^{in} (x) - \partial_t \dr_0 (0,x) \big{)} \exp \big{(} - \tfrac{t}{\eps} \big{)} \, .
\end{equation}

We remark that the initial layer $\dr_0^I \big{(} \frac{t}{\eps^\beta} , x \big{)}$ in \eqref{IL} is, in fact, independent of $\beta > 0$ and $\dr_0^I \big{(} \frac{0}{\eps^\beta} , x \big{)} = - \eps \big{(} \tilde{\dr}^{in} (x) - \partial_t \dr_0 (0,x) \big{)} \rightarrow 0 $ for any given $\tilde{\dr}^{in}(x)$ and $\dr^{in}(x)$ as $\eps \rightarrow 0$. Consequently, the ansatz \eqref{Ansatz-2} is reasonable.

Without loss of generality, we take $\beta = 1$ in the ansatz \eqref{Ansatz-2}. Thus, by substituting \eqref{HF} and \eqref{Ansatz-3} into the system \eqref{Wave-Map}, we derive the  equation satisfied by the remainder $\dr_R^\eps(t,x)$ as follows:
\begin{equation*}
  \begin{aligned}
    & \eps^\frac{3}{2} \partial_{tt} \dr_R^\eps + \eps \partial_{tt} \dr_0 + \sqrt{\eps} \partial_t \dr_R^\eps - \sqrt{\eps} \Delta \dr_R^\eps + \eps \Delta \D(x) \exp \big{(} - \tfrac{t}{\eps} \big{)} \\
    =& - \eps \Big{|} \partial_t \big{(} \dr_0 - \eps \D(x) \exp \big{(} - \tfrac{t}{\eps} \big{)} + \sqrt{\eps} \dr_R^\eps \big{)} \Big{|}^2  \big{(} \dr_0 - \eps \D(x) \exp \big{(} - \tfrac{t}{\eps} \big{)} + \sqrt{\eps} \dr_R^\eps \big{)} \\
    & + \Big{|} \nabla \big{(} \dr_0 - \eps \D(x) \exp \big{(} - \tfrac{t}{\eps} \big{)} + \sqrt{\eps} \dr_R^\eps \big{)} \Big{|}^2  \big{(}- \eps \D(x) \exp \big{(} - \tfrac{t}{\eps} \big{)} + \sqrt{\eps} \dr_R^\eps \big{)} \\
    & + \Big{[} 2 \nabla \dr_0 \cdot \nabla \big{(} - \eps \D(x) \exp ( - \tfrac{t}{\eps} ) + \sqrt{\eps} \dr_R^\eps \big{)} + \big{|} \nabla ( - \eps \D(x) \exp ( - \tfrac{t}{\eps} ) + \sqrt{\eps} \dr_R^\eps  ) \big{|}^2 \Big{]} \dr_0 \, ,
  \end{aligned}
\end{equation*}
which, after multiplication by $\eps^{-\frac{3}{2}}$ is the equation \eqref{Remainder-Eq} we used before.

\section{Uniform Energy Estimates}\label{Sct-Unif-Bounds}

In this section, we will provide, by energy methods,  bounds that are uniform with respect to  small inertia constant $\eps > 0$, for the remainder system \eqref{Remainder-Eq}-\eqref{IC-Remainder}. By Proposition \ref{Prop-HF}, the $\dr_0$, which obeys the heat flow \eqref{HF} into the unit sphere $\mathbb{S}^2$  is regarded as a known quantity in the remainder system \eqref{Remainder-Eq}-\eqref{IC-Remainder}.

To conveniently state our results, we need to introduce the following energy functionals:
\begin{equation*}
  \begin{aligned}
    E_{\eps}(t) &= | \p_t \dre |^2_{H^2} + \big{(} \tfrac{1}{\eps} - 1 \big{)} | \dre |^2_{H^2} + \tfrac{2}{\eps}| \nabla \dre |^2_{H^2} + | \p_t \dre +  \dre |^2_{H^2}\, , \\
    F_{\eps}(t) &= \big{(} \tfrac{1}{\eps} - \tfrac{1}{2} \big{)} | \p_t \dre |^2_{H^2} + \tfrac{1}{2\eps} | \nabla \dre |^2_{H^2}\, .
  \end{aligned}
\end{equation*}
The following lemma provides the claimed uniform energy estimates :
\begin{lemma}\label{Lm-Unif-Bounds}
  Let $\dr^{in} \in\dot{H}^7(\mathbb{R}^3;\mathbb{S}^2)$ and $[0,T]$ be the interval of existence of the the solution of the heat flow with initial data $\dr^{in}$, as provided in Proposition \ref{Prop-HF}.

 For $\tilde \dr^{in}\in H^5$ assume that  there exists a $\eps_0 \equiv \eps_0 (|\nabla \dr^{in} |_{H^6}, |\tilde{\dr}^{in}|_{H^5}, T) \in (0,\frac 12)$ such that for all $\eps \in (0, \eps_0)$  we that  $\dr_R^\eps\in L^\infty ([0,T); {H}^3)$ and $\partial_t \dr^\eps \in L^\infty ([0,T); H^2)$ is a  solution to the remainder system \eqref{Remainder-Eq}-\eqref{IC-Remainder}. Then there exists a positive constant $C = C ( |\nabla \dr^{in}|_{H^6}, |\tilde{\dr}^{in}|_{H^5}, T ) > 0$ such that the inequality
  \begin{equation}\label{Energy-Bounds}
    \frac{\d}{\d t} E_\eps (t) + 3 F_\eps (t) \leq C \big{[} 1 + E_\eps (t) \big{]} \big{[} 1 + \eps E_\eps (t) \big{]}
  \end{equation}
  holds for all $\eps \in (0, \frac{1}{2})$ and $t\in [0,T)$.
\end{lemma}
We remark that the condition $0 < \eps < \frac{1}{2}$  guarantees the relation
$$ \tfrac{1}{2 \eps} < \tfrac{1}{\eps} - 1 < \tfrac{1}{\eps}\, , $$
which makes the energy functionals $E_\eps (t)$ and $F_\eps (t)$  non-negative. Since our goal  is to rigorously analyze the asymptotic behavior as $\eps \rightarrow 0$ for the wave map \eqref{Wave-Map}-\eqref{IC}, the condition $0 < \eps < \frac{1}{2}$ is sufficient.

\begin{proof}
For the convenience of notations, we rewrite the singular terms of the remainder system \eqref{Remainder-Eq} as
\begin{equation}
  \begin{aligned}
  & \left.
      \begin{array}{l}
        \mathcal{S}(\dr^\eps_R) =  -\tfrac{1}{\sqrt{\eps}} \Big{(} \p_{tt} \dr_0 + \Delta \D(x) \exp( -\tfrac{t}{\eps} ) + | \p_t \dr_0 + \D(x) \exp( -\tfrac{t}{\eps} ) |^2 \dr_0 \\
        \qquad\qquad\, + | \nabla \dr_0 |^2 \D(x) \exp( -\tfrac{t}{\eps} ) + 2 \nabla \dr_0 \cdot \nabla \D(x) \exp (-\tfrac{t}{\eps}) \dr_0 \Big{)}
      \end{array}
     \  \right\}\ \mathcal{S}_1 \\
  & \left.
      \begin{array}{l}
      \qquad\qquad \, + \tfrac{1}{\sqrt{\eps}} \big{(} 2 ( \nabla \dr_0 \cdot \nabla \dr^\eps_R ) \dr^\eps_R + | \nabla \dr^\eps_R |^2 \dr_0\big{)}
      \end{array}
     \qquad\qquad\qquad\qquad\quad\quad  \right\}\ \mathcal{S}_2 \\
  & \left.
      \begin{array}{l}
      \qquad\qquad \, + \tfrac{1}{\eps} \big{[} | \nabla \dr_0 |^2 \dr^\eps_R + 2 ( \nabla \dr_0 \cdot \nabla \dr^\eps_R ) \dr_0 \big{]}
      \end{array}
     \qquad\qquad \, \qquad\qquad\qquad\  \,\, \right\}\ \mathcal{S}_3 \\
     & \qquad\quad\, \triangleq \mathcal{S}_1 + \mathcal{S}_2 + \mathcal{S}_3
  \end{aligned}
\end{equation}
and the regular terms as
\begin{equation}
  \begin{aligned}
  & \left.
      \begin{array}{l}
        \mathcal{R}(\dr^\eps_R) = \sqrt{\eps} | \p_t \dr_0 + \D(x) \exp(-\tfrac{t}{\eps}) |^2  \D(x) \exp(-\tfrac{t}{\eps}) + \sqrt{\eps} |\nabla \D(x) |^2 \exp(-\tfrac{2t}{\eps}) \dr_0 \\
        \qquad\qquad\, - \eps^{\frac{3}{2}} | \nabla \D(x) |^2 \exp(-\tfrac{3t}{\eps}) \D(x) + 2\sqrt{\eps} (\nabla \dr_0 \cdot \nabla \D(x) ) \exp(-\tfrac{2t}{\eps}) \D(x)
      \end{array}
     \qquad\quad\, \right\}\ \mathcal{R}_1 \\
  & \left.
      \begin{array}{l}
       \qquad\qquad \,  - 2 \big{(} \p_t \dr_0 + \D(x) \exp(-\tfrac{t}{\eps}) \big{)} \cdot \p_t \dr^\eps_R \dr_0 - | \p_t \dr_0 + \D(x) \exp(-\tfrac{t}{\eps}) |^2 \dr^\eps_R  \\
       \qquad\qquad \, + 2 \eps \big{(} \p_t \dr_0 + \D(x) \exp(-\tfrac{t}{\eps}) \big{)} \cdot \p_t \dr^\eps_R \D(x) \exp(-\tfrac{t}{\eps}) + \eps | \nabla \D |^2 \exp(-\tfrac{2t}{\eps}) \dr^\eps_R \\
       \qquad\qquad \, - 2 ( \nabla \dr_0 \cdot \nabla \D(x) ) \exp(-\tfrac{t}{\eps}) \dr^\eps_R - 2 (\nabla \dr_0 \cdot \nabla \dr^\eps_R ) \D(x) \exp(-\tfrac{t}{\eps}) \\
       \qquad\qquad \, - 2( \nabla \D(x) \cdot \nabla \dr^\eps_R ) \exp(-\tfrac{t}{\eps}) \dr_0 + 2 \eps (\nabla \D(x) \cdot \nabla \dr^\eps_R ) \exp(-\tfrac{2t}{\eps}) \D(x)
      \end{array}
       \qquad \right\}\ \mathcal{R}_2 \\
       & \left.
      \begin{array}{l}
       \qquad\qquad \, + \eps^{\frac{3}{2}} \D(x) \exp(-\tfrac{t}{\eps}) | \p_t \dr^\eps_R |^2 - 2 \sqrt{\eps} \big{(} \p_t \dr_0 + \D(x) \exp(-\tfrac{t}{\eps}) \big{)} \cdot \p_t \dr^\eps_R \dr^\eps_R \\
       \qquad\qquad \, - \sqrt{\eps} | \p_t \dr^\eps_R |^2 \dr_0 - 2\sqrt{\eps} ( \nabla \D(x) \cdot \nabla \dr^\eps_R ) \exp(-\tfrac{t}{\eps}) \dr^\eps_R - \sqrt{\eps} | \nabla \dr^\eps_R |^2 \D(x) \exp(-\tfrac{t}{\eps})
      \end{array}
      \right\}\ \mathcal{R}_3 \\
        & \left.
      \begin{array}{l}
      \qquad\qquad \, -\eps | \p_t \dr^\eps_R |^2 \dr^\eps_R +| \nabla \dr^\eps_R |^2 \dr^\eps_R
      \end{array}
       \qquad\qquad\qquad\qquad\qquad\qquad\qquad\qquad\qquad\qquad\quad\ \  \right\}\ \mathcal{R}_4 \\
     & \qquad\quad \, \triangleq \mathcal{R}_1 + \mathcal{R}_2 + \mathcal{R}_3 + \mathcal{R}_4.
  \end{aligned}
\end{equation}

{\em Step 1. $L^2$-estimates.} Multiplying  the remainder equation \eqref{Remainder-Eq} by $ \p_t \dr^\eps_R $, integrating over $ \mathbb{R}^3 $ and by parts, we  obtain the following equation:
\begin{equation}\label{0orderInit-ptd}
  \begin{aligned}
    \tfrac{1}{2} \tfrac{\d}{\d t} \big{(} | \p_t \dr^\eps_R |^2_{L^2} + \tfrac{1}{\eps} | \nabla \dr^\eps_R |^2_{L^2} \big{)} + \tfrac{1}{\eps} | \p_t \dr^\eps_R |^2_{L^2} = \langle \S (\dre) , \p_t \dre \rangle + \langle \mathcal{R} (\dre) , \p_t \dre \rangle\,.
  \end{aligned}
\end{equation}

{\em (I) Estimates for the singular terms $ \langle \S (\dre) , \p_t \dre \rangle :$ }

For  estimating  $ \langle \S_1 , \p_t \dre \rangle $, we use the  H\"older inequality, the Sobolev embedding theorems, the facts that $ \exp( -\tfrac{t}{\eps} ) \leq 1 $ and $ | \dr_0 | =1 $ to obtain:
\begin{equation*}
  \begin{aligned}
    \tfrac{1}{\sqrt{\eps}} \langle \p_{tt} \dr_0 + \Delta \D(x) \exp(-\tfrac{t}{\eps}) , \p_t \dre \rangle
    & \lesssim ( | \p_{tt} \dr_0 |_{L^\infty_t L^2_x} + | \Delta \D(x) |_{L^2}) | \tfrac{ \p_t \dre }{ \sqrt{\eps} } |_{L^2}
  \end{aligned}
\end{equation*}
and
\begin{equation*}
  \begin{aligned}
    & \tfrac{1}{\sqrt{\eps}} \langle | \nabla \dr_0 |^2 \D(x) \exp( -\tfrac{t}{\eps} ) + 2 \nabla \dr_0 \cdot \nabla \D(x) \exp (-\tfrac{t}{\eps}) \dr_0 , \p_t \dre \rangle \\
    \lesssim & \big{(}| \nabla \dr_0 |^2_{L^\infty} | \D(x) |_{L^2} + | \nabla \dr_0 |_{L^\infty} | \nabla \D(x) |_{L^2} \big{)} | \tfrac{\p_t \dre}{\sqrt{\eps}} |_{L^2} \\
    \lesssim & \big{(} | \nabla \dr_0 |^2_{L^\infty_t H^2_x} + | \nabla \dr_0 |_{L^\infty_t H^2_x} \big{)} | \D(x) |_{H^1} | \tfrac{\p_t \dre}{\sqrt{\eps}} |_{L^2}\,.
  \end{aligned}
\end{equation*}

Similarly:
\begin{equation*}
  \begin{aligned}
    \tfrac{1}{\sqrt{\eps}} \langle | \p_t \dr_0 + \D(x) \exp( -\tfrac{t}{\eps} ) |^2 \dr_0 , \p_t \dre \rangle
    & \lesssim | \p_t \dr_0 + \D(x) \exp( -\tfrac{t}{\eps} ) |^2_{L^4} | \tfrac{ \p_t \dre }{\sqrt{\eps}} |_{L^2}\\
    & \lesssim \big{(} | \p_t \dr_0 |^2_{L^\infty_t L^4_x} + | \D(x) |^2_{L^4} \big{)} | \tfrac{ \p_t \dre }{\sqrt{\eps}} |_{L^2}\\
    & \lesssim \big{(} | \p_t \dr_0 |^2_{L^\infty_t H^1_x} + | \D(x) |^2_{H^1} \big{)} | \tfrac{ \p_t \dre }{\sqrt{\eps}} |_{L^2}\, .
  \end{aligned}
\end{equation*}
Summarizing, we estimate $ \langle \S_1 , \p_t \dre \rangle $ as follows:
\begin{equation}\label{0orderS1}
  \begin{aligned}
    \langle \S_1 , \p_t \dre \rangle \lesssim & \Big{[} | \p_{tt} \dr_0 |_{L^\infty_t L^2_x} + | \p_t \dr_0 |^2_{L^\infty_t H^1_x} + | \D(x) |^2_{H^1} \\
    &+ ( | \nabla \dr_0 |^2_{L^\infty_t H^2_x} +1 ) | \D(x) |_{H^2} \Big{]} \big{|} \tfrac{\p_t \dre}{\sqrt{\eps}} \big{|}_{ L^2_x}\, .
  \end{aligned}
\end{equation}

It is easy to derive the estimates of $ \langle \S_2 , \p_t \dre \rangle $ and $ \langle \S_3 , \p_t \dre \rangle $ as follows:
\begin{equation}\label{0orderS2}
  \begin{aligned}
 \langle \S_2 , \p_t \dre \rangle & = \tfrac{1}{\sqrt{\eps}} \langle 2 ( \nabla \dr_0 \cdot \nabla \dr^\eps_R ) \dr^\eps_R + | \nabla \dr^\eps_R |^2 \dr_0 , \p_t \dre \rangle\\
  & \lesssim ( |\nabla \dr_0 |_{L^\infty} | \dre |_{L^\infty} | \nabla \dre |_{L^2} + | \nabla \dre |^2_{L^4} ) | \tfrac{\p_t \dre}{\sqrt{\eps}} |_{L^2}\\
  & \lesssim ( | \nabla \dr_0 |_{L^\infty_t H^2_x} + 1 ) | \dre |^2_{H^2} | \tfrac{\p_t \dre}{\sqrt{\eps}} |_{L^2}
  \end{aligned}
\end{equation}
and
\begin{equation}\label{0orderS3}
  \begin{aligned}
    \langle \S_3 , \p_t \dre \rangle & = \tfrac{1}{\eps} \langle | \nabla \dr_0 |^2 \dr^\eps_R + 2 ( \nabla \dr_0 \cdot \nabla \dr^\eps_R ) \dr_0 , \p_t \dre \rangle \\
    & \lesssim ( | \nabla \dr_0 |^2_{L^\infty} | \tfrac{\dre}{\sqrt{\eps}}|_{L^2} + | \nabla \dr_0 |_{L^\infty} | \tfrac{\nabla \dre}{\sqrt{\eps}}|_{L^2} ) | \tfrac{\p_t \dre}{\sqrt{\eps}} |_{L^2} \\
    & \lesssim \big{(} | \nabla \dr_0 |_{L^\infty_t H^2_x} + | \nabla \dr_0 |^2_{L^\infty_t H^2_x} \big{)} \big{(} | \tfrac{\dre}{\sqrt{\eps}} |_{L^2} + | \tfrac{\nabla \dre}{\sqrt{\eps}}|_{L^2} \big{)} | \tfrac{\p_t \dre}{\sqrt{\eps}} |_{L^2}\, .
  \end{aligned}
\end{equation}
Hence we have the estimate of $ \langle \S(\dre) , \p_t \dre \rangle $ by combining the inequalities \eqref{0orderS1}, \eqref{0orderS2} and \eqref{0orderS3}:
\begin{equation}\label{0orderS-ptd}
  \langle S(\dre) , \p_t \dre \rangle
  \leq C_{11} \big{(} 1 + | \dre |^2_{H^2} + | \tfrac{\dre}{\sqrt{\eps}}|_{L^2} + | \tfrac{\nabla \dre}{\sqrt{\eps}}|_{L^2} \big{)} | \tfrac{\p_t \dre}{\sqrt{\eps}} |_{L^2}\, ,
\end{equation}
where the constant
$$ C_{11} = C \big{[} | \p_{tt} \dr_0 |_{L^\infty_t L^2_x} + | \p_t \dr_0 |^2_{L^\infty_t H^1_x} + | \D(x) |^2_{H^1} + ( | \nabla \dr_0 |^2_{L^\infty_t H^2_x} +1 ) ( | \D(x) |_{H^2} + 1 ) \big{]} > 0 $$
for some computable positive constant $ C > 0$.

{ \em (II) Estimates for the regular terms $ \langle \mathcal{R}(\dre) , \p_t \dre \rangle $:}

\smallskip
We have divided the regular terms $ \mathcal{R}(\dre) $ into four parts, which we will estimate separately.

\bigskip
\underline{The estimate of $ \langle \mathcal{R}_1 , \p_t \dre \rangle $}.
By the H\"older inequality and the Sobolev embedding theorems we have
\begin{equation*}
  \begin{aligned}
    & \sqrt{\eps} \langle | \p_t \dr_0 + \D(x) \exp(-\tfrac{t}{\eps}) |^2 \D(x) \exp(-\tfrac{t}{\eps}) , \p_t \dre \rangle \\
    \lesssim & \sqrt{\eps} | \p_t \dr_0 + \D(x) \exp(-\tfrac{t}{\eps}) |^2_{L^\infty} |\D(x)|_{L^2} | \p_t \dre |_{L^2} \\
    \lesssim & \sqrt{\eps} \big{(} | \p_t \dr_0 |^2_{L^\infty_t H^2_x} + | \D(x) |^2_{H^2} \big{)} | \D(x) |_{L^2} | \p_t \dre |_{L^2}
  \end{aligned}
\end{equation*}
and
\begin{equation*}
  \begin{aligned}
    \sqrt{\eps} \langle (\nabla \dr_0 \nabla \D(x) ) \exp(-\tfrac{2t}{\eps}) \D(x) , \p_t \dre \rangle
    & \lesssim \sqrt{\eps} | \nabla \dr_0 |_{L^\infty} | \D(x) |_{L^\infty} | \nabla \D(x) |_{L^2} | \p_t \dre |_{L^2}\\
    & \lesssim \sqrt{\eps} | \nabla \dr_0 |_{L^\infty_t H^2_x} | \D(x) |^2_{H^2} | \p_t \dre |_{L^2}\, ,
  \end{aligned}
\end{equation*}
where we have used the fact that $ \exp(-\tfrac{t}{\eps}) \leq 1 $.

The other two terms in $ \langle \mathcal{R}_1 , \p_t \dre \rangle $ are similarly estimated, as follows:
\begin{equation*}
  \begin{aligned}
    \sqrt{\eps} \langle | \nabla \D(x) |^2 \exp(-\tfrac{2t}{\eps}) \dr_0 , \p_t \dre \rangle
    & \lesssim \sqrt{\eps} |\nabla \D(x) |^2_{L^4} | \p_t \dre |_{L^2} \\
    & \lesssim \sqrt{\eps} |\D(x)|^2_{H^2} | \p_t \dre |_{L^2}\, , \\
    \eps^{\frac{3}{2}} \langle | \nabla \D(x) |^2 \exp(-\tfrac{3t}{\eps}) \D(x) , \p_t \dre \rangle
    & \lesssim \eps^{\frac{3}{2}} | \D(x) |_{L^\infty} | \nabla \D(x) |^2_{L^4} | \p_t \dre |_{L^2} \\
    & \lesssim \eps^{\frac{3}{2}} | \D(x) |^3_{H^2} | \p_t \dre |_{L^2}\, .
  \end{aligned}
\end{equation*}
For any small enough $ \eps $  such that $ \eps \in ( 0 , \tfrac{1}{2} ] $, we have $ \eps^{\frac{3}{2}} \leq \eps^{\frac{1}{2}}$. So from the above inequalities, we obtain the following estimate:
\begin{equation}\label{0orderR1}
  \begin{aligned}
    \langle \mathcal{R}_1 , \p_t \dre \rangle
    \lesssim \sqrt{\eps} \big{[} ( | \p_t \dr_0 |^2_{L^\infty_t H^2_x} + | \nabla \dr_0 |^2_{L^\infty_t H^2_x} + | \D(x) |^2_{H^2} ) | \D(x) |_{H^2} + | \D(x) |^2_{H^2} \big{]} | \p_t \dre |_{L^2}\, .
  \end{aligned}
\end{equation}

\underline{The estimate of $ \langle \mathcal{R}_2 , \p_t \dre \rangle $}.
We have the following estimates for the first three terms in $ \langle \mathcal{R}_2 , \p_t \dre \rangle $:
\begin{equation*}
  \begin{aligned}
    \langle | \p_t \dr_0 + \D(x) \exp(-\tfrac{t}{\eps}) |^2 \dr^\eps_R , \p_t \dre \rangle
    & \lesssim | \p_t \dr_0 + \D(x) \exp(-\tfrac{t}{\eps}) |^2_{L^\infty} | \dre |_{L^2} | \p_t \dre |_{L^2} \\
    & \lesssim \big{(} | \p_t \dr_0 |^2_{L^\infty_t H^2_x} + | \D(x) |^2_{H^2} \big{)} | \dre |_{L^2} | \p_t \dre |_{L^2}\, ,\\
    \langle \big{(} \p_t \dr_0 + \D(x) \exp(-\tfrac{t}{\eps}) \big{)} \cdot \p_t \dre \dr_0 , \p_t \dre \rangle
    & \lesssim \big{(} | \p_t \dr_0 |_{L^\infty} + | \D(x) |_{L^\infty} \big{)} | \p_t \dre |^2_{L^2} \\
    & \lesssim \big{(} | \p_t \dr_0 |_{L^\infty_t H^2_x} + | \D(x) |_{H^2} \big{)} | \p_t \dre |^2_{L^2}
  \end{aligned}
\end{equation*}
and
\begin{equation*}
  \begin{aligned}
    &\eps \langle \big{(} \p_t \dr_0 + \D(x) \exp(-\tfrac{t}{\eps}) \big{)} \cdot \D(x) \exp(-\tfrac{t}{\eps}) \p_t \dr^\eps_R , \p_t \dre \rangle \\
    \lesssim &  \eps | \p_t \dr_0 + \D(x) \exp(-\tfrac{t}{\eps}) |_{L^\infty} | \D(x) |_{L^\infty} | \p_t \dre |^2_{L^2} \\
    \lesssim & \eps \big{(} | \p_t \dr_0 |_{L^\infty_t H^2_x} + | \D(x) |_{H^2} \big{)} | \D(x) |_{H^2} | \p_t \dre |^2_{L^2}\, , \\
  \end{aligned}
\end{equation*}
where we have used the H\"older inequality, the Sobolev embedding theorems and the fact that $ | \dr_0 | = 1 $.

As for the following three terms, one can easily obtain:

\begin{equation*}
  \begin{aligned}
    \eps \langle | \nabla \D(x) |^2 \exp(-\tfrac{2t}{\eps}) \dre , \p_t \dre \rangle
    & \lesssim \eps | \nabla \D(x) |^2_{L^\infty} | \dre |_{L^2}| \p_t \dre |_{L^2} \\
    & \lesssim \eps | \D(x) |^2_{H^3} | \dre |_{L^2}| \p_t \dre |_{L^2}\, , \\
    \langle (\nabla \dr_0 \cdot \nabla \D(x) ) \exp(-\tfrac{t}{\eps}) \dre , \p_t \dre \rangle
    & \lesssim | \nabla \dr_0 |_{L^\infty} | \nabla \D(x) |_{L^\infty} | \dre |_{L^2} | \p_t \dre |_{L^2} \\
    & \lesssim | \nabla \dr_0 |_{L^\infty_t H^2_x} | \D(x) |_{H^3} | \dre |_{L^2} | \p_t \dre |_{L^2}\, , \\
    \langle (\nabla \dr_0 \cdot \nabla \dre ) \D(x) \exp(-\tfrac{t}{\eps}) , \p_t \dre \rangle
    & \lesssim | \nabla \dr_0 |_{L^\infty} | \D(x) |_{L^\infty} | \nabla \dre |_{L^2} | \p_t \dre |_{L^2} \\
    & \lesssim | \nabla \dr_0 |_{L^\infty_t H^2_x} | \D(x) |_{H^2} | \nabla \dre |_{L^2} | \p_t \dre |_{L^2}\, , \\
  \end{aligned}
\end{equation*}
where we have used the H\"older inequality, the Sobolev embedding theorems and the bound $ \exp(-\tfrac{t}{\eps}) \leq 1 $.

Similarly as before, we  estimate the last two terms, as follows:
\begin{equation*}
  \begin{aligned}
    \langle (\nabla \D(x) \cdot \nabla \dre ) \exp(-\tfrac{t}{\eps}) \dr_0 , \p_t \dre \rangle
    & \lesssim | \nabla \D(x) |_{L^\infty} | \nabla \dre |_{L^2} | \p_t \dre |_{L^2} \\
    & \lesssim | \D(x) |_{H^3} | \nabla \dre |_{L^2} | \p_t \dre |_{L^2}\, , \\
    \eps \langle ( \nabla \D(x) \cdot \nabla \dre ) \exp(-\tfrac{2t}{\eps}) \D(x) , \p_t \dre \rangle
    & \lesssim \eps | \nabla \D(x) |_{L^\infty} | \D(x) |_{L^\infty} | \nabla \dre |_{L^2} | \p_t \dre |_{L^2} \\
    & \lesssim \eps | \D(x) |^2_{H^3} | \nabla \dre |_{L^2} | \p_t \dre |_{L^2} \, .
  \end{aligned}
\end{equation*}

Combining the above estimates and using that $ \eps \in ( 0, \tfrac{1}{2} ] $, we get
\begin{equation}\label{0orderR2}
  \begin{aligned}
    \langle \mathcal{R}_2 , \p_t \dre \rangle
    \lesssim & \big{(} 1 + | \p_t \dr_0 |_{L^\infty_t H^2_x} + | \nabla \dr_0 |_{L^\infty_t H^2_x} + | \D(x) |_{H^3} \big{)} \big{(} | \p_t \dr_0 |_{L^\infty_t H^2_x} + | \D(x) |_{H^3} \big{)} \\
     & \times \big{(} | \dre |_{L^2} + | \nabla \dre |_{L^2} + | \p_t \dre|_{L^2} \big{)} | \p_t \dre |_{L^2}\, .
  \end{aligned}
\end{equation}

\underline{The estimate of $ \langle \mathcal{R}_3 , \p_t \dre \rangle $}.One can easily  derive the following estimates
\begin{equation*}
  \begin{aligned}
    & \sqrt{\eps} \langle ( \p_t \dr_0 + \D(x) \exp(-\tfrac{t}{\eps}) ) \cdot \p_t \dre \dre , \p_t \dre \rangle
    \lesssim \sqrt{\eps} \big{(} | \p_t \dr_0 |_{L^\infty_t H^2_x} + | \D(x) |_{H^2} \big{)} | \dre |_{H^2} | \p_t \dre|^2_{L^2}\, , \\
    & \sqrt{\eps} \langle (\nabla \D(x) \cdot \nabla \dre) \exp(-\tfrac{t}{\eps}) \dre , \p_t \dre \rangle
    \lesssim \sqrt{\eps} | \D(x) |_{H^3} | \dre |_{H^2} | \nabla \dre |_{L^2} | \p_t \dre |_{L^2}\, ,
  \end{aligned}
\end{equation*}
by using the H\"older inequality and the Sobolev embedding theorems. Recalling that $ |\dr_0| = 1 $ and using the Sobolev embeddings $H^1 \hookrightarrow L^4$ and $H^2 \hookrightarrow L^\infty$ we get
\begin{equation*}
  \begin{aligned}
    \eps^{\frac{3}{2}} \langle \D(x) \exp(-\tfrac{t}{\eps}) | \p_t \dre |^2 , \p_t \dre \rangle
    & \lesssim \eps^{\frac{3}{2}} | \D(x) |_{L^\infty} | \p_t \dre |^2_{L^4} | \p_t \dre |_{L^2} \\
    & \lesssim \eps^{\frac{3}{2}} | \D(x) |_{H^2} | \p_t \dre |^2_{H^1} | \p_t \dre |_{L^2}\, , \\
    \sqrt{\eps} \langle | \p_t \dre |^2 \dr_0 , \p_t \dre \rangle
    & \lesssim \sqrt{\eps} | \p_t \dre |^2_{L^4} | \p_t \dre |_{L^2} \\
    & \lesssim \sqrt{\eps} | \p_t \dre |^2_{H^1} | \p_t \dre |_{L^2}\, , \\
    \sqrt{\eps} \langle | \nabla \dre |^2 \D(x) \exp({-\tfrac{t}{\eps}}) , \p_t \dre \rangle
    & \lesssim \sqrt{\eps} | \D(x) |_{L^\infty} | \nabla \dre |^2_{L^4} | \p_t \dre |_{L^2} \\
    & \lesssim \sqrt{\eps} | \D(x) |_{H^2} | \nabla \dre |^2_{H^1} | \p_t \dre |_{L^2}\, ,
  \end{aligned}
\end{equation*}

The above estimates immediately give the bound on  $ \langle \mathcal{R}_3 , \p_t \dre \rangle $:
\begin{equation}\label{0orderR3}
  \begin{aligned}
    \langle \mathcal{R}_3 , \p_t \dre \rangle
    \lesssim & \sqrt{\eps}  \big{(} 1 + | \p_t \dr_0 |_{L^\infty_t H^2_x} + | \D(x) |_{H^3} \big{)} \big{(} | \p_t \dre |^2_{H^1} + | \nabla \dre |^2_{H^1}  \\
    & + | \dre |_{H^2} | \nabla \dre |_{L^2} + | \dre |_{H^2} | \p_t \dre|_{L^2} \big{)} | \p_t \dre |_{L^2}\\
    \lesssim & \sqrt{\eps}  \big{(} 1 + | \p_t \dr_0 |_{L^\infty_t H^2_x} + | \D(x) |_{H^3} \big{)} \big{(} | \p_t \dre |^2_{H^1} + | \dre |^2_{H^2} \big{)} | \p_t \dre |_{L^2} \, .
  \end{aligned}
\end{equation}

\underline{The estimate of $ \langle \mathcal{R}_4 , \p_t \dre \rangle $}.
The first term in $ \langle \mathcal{R}_4 , \p_t \dre \rangle $ can be bounded as
\begin{equation*}
  \begin{aligned}
    \eps \langle | \p_t \dre |^2 \dre , \p_t \dre \rangle
    \leq & \eps | \dre |_{L^\infty} | \p_t \dre |^2_{L^4} | \p_t \dre |_{L^2} \\
    \lesssim & \eps | \dre |_{H^2} | \p_t \dre |^2_{H^1} | \p_t \dre |_{L^2}
  \end{aligned}
\end{equation*}
by using the H\"older inequality and the Sobolev embedding theorems. The other term can be bounded in a similar way:
\begin{equation*}
  \begin{aligned}
    \langle | \nabla \dre |^2 \dre , \p_t \dre \rangle
    & \lesssim | \dre |_{L^\infty} | \nabla \dre |^2_{L^4} | \p_t \dre |_{L^2} \\
    & \lesssim | \dre |_{H^2} | \nabla \dre |^2_{H^1} | \p_t \dre |_{L^2}\, .
  \end{aligned}
\end{equation*}

Hence we obtain the estimate of $ \langle \mathcal{R}_4 , \p_t \dre \rangle $ as follows:
\begin{equation}\label{0orderR4}
  \langle \mathcal{R}_4 , \p_t \dre \rangle
  \lesssim \big{(} \eps | \p_t \dre |^2_{H^1} + | \nabla \dre |^2_{H^1}\big{)} | \dre |_{H^2} | \p_t \dre |_{L^2}\,.
\end{equation}

\bigskip
Summing up the inequalities \eqref{0orderR1}, \eqref{0orderR2}, \eqref{0orderR3} and \eqref{0orderR4}, we get
\begin{equation}\label{0orderR-ptd}
  \begin{aligned}
    \langle \mathcal{R}(\dre) , \p_t \dre \rangle
    & \leq C_{12} \Big{[} \sqrt{\eps} + | \dre |_{H^1} + | \p_t \dre |_{L^2} + \sqrt{\eps} | \p_t \dre |^2_{H^1} + \sqrt{\eps} | \dre |^2_{H^2} \\
    & \quad + | \dre |_{H^2} \big{(} \eps | \p_t \dre |^2_{H^1} + | \nabla \dre |^2_{H^1} \big{)}\Big{]} | \p_t \dre |_{L^2}\, ,
  \end{aligned}
\end{equation}
where
$$ C_{12} = C \big{(} 1 + | \p_t \dr_0 |^2_{L^\infty_t H^2_x} + | \nabla \dr_0 |^2_{L^\infty_t H^2_x} + | \D(x) |^2_{H^3} \big{)} \big{(} 1 + | \p_t \dr_0 |_{L^\infty_t H^2_x} + | \D(x) |_{H^3} \big{)} > 0 \,,$$
and $C $ is a positive computable constant.

Therefore, plugging the estimates \eqref{0orderS-ptd} and \eqref{0orderR-ptd} into the equality \eqref{0orderInit-ptd}, we have
\begin{equation}\label{0order-ptd}
  \begin{aligned}
    & \tfrac{1}{2} \tfrac{\d}{\d t} \big{(} | \p_t \dr^\eps_R |^2_{L^2} + \tfrac{1}{\eps} | \nabla \dr^\eps_R |^2_{L^2} \big{)} + \tfrac{1}{\eps} | \p_t \dr^\eps_R |^2_{L^2} \\
    \leq & C_{1} \Big{\{} \big{(} 1 + | \dre |^2_{H^2} + | \tfrac{\dre}{\sqrt{\eps}}|_{L^2} + | \tfrac{\nabla \dre}{\sqrt{\eps}}|_{L^2} \big{)} | \tfrac{\p_t \dre}{\sqrt{\eps}} |_{L^2} +  \Big{[} \sqrt{\eps} + | \dre |_{H^1} + | \p_t \dre |_{L^2} \\
    & \quad + \sqrt{\eps} | \p_t \dre |^2_{H^1} + \sqrt{\eps} | \dre |^2_{H^2} + | \dre |_{H^2} \big{(} \eps | \p_t \dre |^2_{H^1} + | \nabla \dre |^2_{H^1} \big{)}\Big{]} | \p_t \dre |_{L^2} \Big{\}} \, ,
  \end{aligned}
\end{equation}
where the constant
$$ C_1 = C ( 1 + |\D(x)|^3_{H^3} + |\partial_{tt} \dr_0|_{L^\infty(0,T;L^2)} + |\partial_t \dr_0|^3_{L^\infty(0,T;H^2)} + |\nabla \dr_0|^4_{L^\infty(0,T;H^2) } ) > 0  $$
for some computable positive constant $C$.

\bigskip
{\em (III) Estimates of the norm $|\dr_R^\eps|_{L^2}$:}

Observing that the norm $|\dre|_{L^2}$ appearing on the right hand side of \eqref{0order-ptd} is not yet controlled, we need additional work to estimate $|\dre|_{L^2}$. In order to do this it is natural to multiply  the equation of  the reminder term \eqref{Remainder-Eq} by  $ \dre $, integrate over $ \mathbb{R}^3 $ and  by parts, and use the identity:
\begin{equation*}
  \begin{aligned}
     \langle \p_{tt} \dre , \dre \rangle
    =& \tfrac{\d }{\d t} \langle \p_t \dre , \dre \rangle - | \p_t \dre |^2_{L^2} \\
    =&  \tfrac{1}{2} \tfrac{\d }{\d t}\big{(} | \p_t \dre + \dre |^2_{L^2} - | \p_t \dre |^2_{L^2} - | \dre |^2_{L^2} \big{)} - | \p_t \dre |^2_{L^2} \, ,
  \end{aligned}
\end{equation*}
in order to get
\begin{equation}\label{0orderInit-d}
  \begin{aligned}
    \tfrac{1}{2} \tfrac{\d }{\d t} & \Big{[} | \p_t \dre + \dre |^2_{L^2} + \big{(} \tfrac{1}{\eps} - 1 \big{)} | \dre |^2_{L^2} - | \p_t \dre |^2_{L^2} \Big{]} - | \p_t \dre |^2_{L^2} + \tfrac{1}{\eps} | \nabla \dre |^2_{L^2} \\
   & = \langle \S (\dre) , \dre \rangle + \langle \mathcal{R} (\dre) , \dre \rangle\, .
  \end{aligned}
\end{equation}
Using the estimates \eqref{0orderS-ptd} and \eqref{0orderR-ptd} previously derived for bounding the terms  $ \langle \S (\dre) , \p_t \dre \rangle $ and $ \langle \mathcal{R} (\dre) , \p_t \dre \rangle $, we can analogously estimate the terms $ \langle \S (\dre) , \dre \rangle $ and $ \langle \mathcal{R} (\dre) , \dre \rangle $ as follows:

\begin{equation}\label{0orderS-d}
  \langle S(\dre) , \dre \rangle
  \leq C_{11} \big{(} 1 + | \dre |^2_{H^2} + | \tfrac{\dre}{\sqrt{\eps}}|_{L^2} + | \tfrac{\nabla \dre}{\sqrt{\eps}}|_{L^2} \big{)} | \tfrac{\dre}{\sqrt{\eps}} |_{L^2}\, ,
\end{equation}
and
\begin{equation}\label{0orderR-d}
  \begin{aligned}
    \langle \mathcal{R} (\dre) , \dre \rangle
    & \leq C_{12} \Big{[} \sqrt{\eps} + | \dre |_{H^1} + | \p_t \dre |_{L^2} + \sqrt{\eps} | \p_t \dre |^2_{H^1} + \sqrt{\eps} | \dre |^2_{H^2} \\
    & \quad + | \dre |_{H^2} \big{(} \eps | \p_t \dre |^2_{H^1} + | \nabla \dre |^2_{H^1} \big{)}\Big{]} | \dre |_{L^2}\, .
  \end{aligned}
\end{equation}
So plugging \eqref{0orderS-d} and \eqref{0orderR-d} into \eqref{0orderInit-d} we obtain:
\begin{equation}\label{0order-d}
  \begin{aligned}
    & \tfrac{1}{2} \tfrac{\d }{\d t} \Big{[} | \p_t \dre + \dre |^2_{L^2} + \big{(} \tfrac{1}{\eps} - 1 \big{)} | \dre |^2_{L^2} - | \p_t \dre |^2_{L^2} \Big{]} - | \p_t \dre |^2_{L^2} + \tfrac{1}{\eps} | \nabla \dre |^2_{L^2} \\
    \leq & C_{1} \Big{\{} \big{(} 1 + | \dre |^2_{H^2} + | \tfrac{\dre}{\sqrt{\eps}}|_{L^2} + | \tfrac{\nabla \dre}{\sqrt{\eps}}|_{L^2} \big{)} | \tfrac{\dre}{\sqrt{\eps}} |_{L^2} +  \Big{[} \sqrt{\eps} + | \dre |_{H^1} + | \p_t \dre |_{L^2} \\
    & \quad + \sqrt{\eps} | \p_t \dre |^2_{H^1} + \sqrt{\eps} | \dre |^2_{H^2} + | \dre |_{H^2} \big{(} \eps | \p_t \dre |^2_{H^1} + | \nabla \dre |^2_{H^1} \big{)}\Big{]} | \dre |_{L^2} \Big{\}} \, .
  \end{aligned}
\end{equation}

Multiplying  the inequality \eqref{0order-d} by $ \tfrac{1}{2} $ and then adding it to the inequality \eqref{0order-ptd}, we get the $L^2$-energy estimate:
\begin{equation}\label{0orderEst}
  \begin{aligned}
      \tfrac{1}{4} \tfrac{\d }{\d t} & \Big{[} | \p_t \dre |^2_{L^2} +  | \p_t \dre + \dre |^2_{L^2} + \big{(} \tfrac{1}{\eps} - 1 \big{)} | \dre |^2_{L^2} \\
     &+ \tfrac{2}{\eps} | \nabla \dre |^2_{L^2} \Big{]}  + (\tfrac{1}{\eps}-\tfrac{1}{2}) | \p_t \dre |^2_{L^2} + \tfrac{1}{2\eps} | \nabla \dre |^2_{L^2} \\
    &\leq  C_{1} \Big{\{} \big{(} 1 + | \dre |^2_{H^2} + | \tfrac{\dre}{\sqrt{\eps}} |_{L^2} + | \tfrac{\nabla \dre}{\sqrt{\eps}}|_{L^2} \big{)} \big{(} | \tfrac{\p_t \dre}{\sqrt{\eps}} |_{L^2} + | \tfrac{\dre}{\sqrt{\eps}} |_{L^2} ) \\
    & \quad+ \Big{[} \sqrt{\eps} + | \dre |_{H^1} + | \p_t \dre |_{L^2} + \sqrt{\eps} | \p_t \dre |^2_{H^1} + \sqrt{\eps} | \dre |^2_{H^2} \\
    & \quad + | \dre |_{H^2} \big{(} \eps | \p_t \dre |^2_{H^1} + | \nabla \dre |^2_{H^1} \big{)}\Big{]}
    \big{(} | \p_t \dre |_{L^2} + | \dre |_{L^2} \big{)} \Big{\}} \, ,
  \end{aligned}
\end{equation}
where the constant
$$ C_1 = C ( 1 + |\D(x)|^3_{H^3} + |\partial_{tt} \dr_0|_{L^\infty(0,T;L^2)} + |\partial_t \dr_0|^3_{L^\infty(0,T;H^2)} + |\nabla \dr_0|^4_{L^\infty(0,T;H^2) } ) > 0  $$
for some computable positive constant $C$.

\bigskip
{\em Step 2. Higher order estimates. }
In order to use the inequality \eqref{0orderEst} we also need a higher order estimate. To obtain this we take $ \nabla^k ( k = 1, 2 )$ in the equation \eqref{Remainder-Eq}, we multiply it  by $ \nabla^k \p_t \dre $, integrate over $ \mathbb{R}^3 $ and  by parts, thus obtaining the following equality
\begin{equation}\label{korderInit-ptd}
  \begin{aligned}
    & \tfrac{1}{2} \tfrac{\d}{\d t} \big{(} | \p_t \nabla^k \dr^\eps_R |^2_{L^2} + \tfrac{1}{\eps} | \nabla^{k+1} \dr^\eps_R |^2_{L^2} \big{)} + \tfrac{1}{\eps} | \p_t \nabla^k \dr^\eps_R |^2_{L^2}\\
     = & \langle \nabla^k \S (\dre) , \p_t \nabla^k \dre \rangle + \langle \nabla^k \mathcal{R} (\dre) , \p_t \nabla^k \dre \rangle
  \end{aligned}
\end{equation}

\bigskip
{ \em (I) Estimates of the singular terms $ \langle \nabla^k \S(\dre) , \p_t \nabla^k \dre \rangle $:}

The singular terms  can be divided into three parts: $ \langle \nabla^k \S_i , \p_t \nabla^k \dre \rangle (i = 1, 2, 3)$ which we estimate separately.

\underline{For the term $ \langle \nabla^k \S_1 , \p_t \nabla^k \dre \rangle $}, by using the H\"older inequality and the Sobolev embedding theorems, we obtain:
\begin{equation*}
  \begin{aligned}
    \tfrac{1}{\sqrt{\eps}} \langle \nabla^k \big{(} \p_{tt} \dr_0 + \Delta \D(x) \exp(-\tfrac{t}{\eps})\big{)} , \p_t \nabla^k \dre \rangle
    \lesssim ( | \p_{tt} \dr_0 |_{L^\infty_t H^2_x} + | \D(x) |_{H^4} ) \big{|} \tfrac{\p_t \dre}{\sqrt{\eps}} \big{|}_{H^2}
  \end{aligned}
\end{equation*}
and
\begin{equation*}
  \begin{aligned}
    & \tfrac{1}{\sqrt{\eps}} \langle \nabla^k ( | \p_t \dr_0 + \D(x) \exp(-\tfrac{t}{\eps})|^2 \dr_0 ) , \p_t \nabla^k \dre \rangle\\
    \lesssim & \tfrac{1}{\sqrt{\eps}}  \sum_{ \substack{i+j+e=k \\ e \geqslant 1} } | \langle \nabla^i(\p_t \dr_0 + \D(x)) \nabla^j(\p_t \dr_0 + \D(x)) \nabla^e \dr_0 , \p_t \nabla^k \dre \rangle|\\
    & + \tfrac{1}{\sqrt{\eps}}  \sum_{ \substack{i+j=k} } | \langle \nabla^i(\p_t \dr_0 + \D(x)) \nabla^j(\p_t \dr_0 + \D(x)) \dr_0 , \p_t \nabla^k \dre \rangle|\\
    \lesssim & ( 1 + | \nabla \dr_0 |_{L^\infty_t H^2_x} ) ( | \p_t \dr_0 |^2_{L^\infty_t H^2_x} + | \D(x) |^2_{H^2} ) \big{|} \tfrac{\p_t \dre}{\sqrt{\eps}} \big{|}_{H^2}\, .
  \end{aligned}
\end{equation*}
Similarly as for estimating  $ \tfrac{1}{\sqrt{\eps}} \langle \nabla^k ( | \p_t \dr_0 + \D(x) \exp(-\tfrac{t}{\eps})|^2 \dr_0 ) , \p_t \nabla^k \dre \rangle $, we can easily get the following estimates:
\begin{equation*}
  \begin{aligned}
    & \tfrac{1}{\sqrt{\eps}} \langle \nabla^k \big{(} | \nabla \dr_0 |^2 \D(x) \exp(-\tfrac{t}{\eps}) \big{)} , \p_t \nabla^k \dre \rangle \\
    \lesssim & \tfrac{1}{\sqrt{\eps}} \sum_{ \substack{i+j+e=k}} |\langle \nabla^{i+1} \dr_0 \nabla^{j+1} \dr_0 \nabla^e \D(x) , \p_t \nabla^k \dre \rangle| \\
    \lesssim & | \nabla \dr_0 |^2_{L^\infty_t H^2_x} | \D(x) |_{H^3} | \tfrac{\p_t \dre}{\sqrt{\eps}} |_{H^2}
  \end{aligned}
\end{equation*}
and
\begin{equation*}
  \begin{aligned}
    & \tfrac{1}{\sqrt{\eps}} \langle \nabla^k \big{(} \nabla \dr_0 \cdot \nabla \D(x) \exp(-\tfrac{t}{\eps}) \dr_0 \big{)} , \p_t \nabla^k \dre \rangle \\
    \lesssim & \tfrac{1}{\sqrt{\eps}} \sum_{ \substack{i+j+e=k \\ e \geqslant 1}} |\langle \nabla^{i+1} \dr_0 \cdot \nabla^{j+1} \D(x)\nabla^e \dr_0 , \p_t \nabla^k \dre \rangle| \\
    & + \tfrac{1}{\sqrt{\eps}} \sum_{ \substack{i+j=k}} |\langle \nabla^{i+1} \dr_0 \cdot \nabla^{j+1} \D(x) \dr_0 , \p_t \nabla^k \dre \rangle|  \\
    \lesssim & (1 + | \nabla \dr_0 |_{L^\infty_t H^2_x}) | \nabla \dr_0 |_{L^\infty_t H^3_x} | \D(x) |_{H^3} | \tfrac{\p_t \dre}{\sqrt{\eps}} |_{H^2} \, . \\
  \end{aligned}
\end{equation*}
Thus we have the estimate of $ \langle \nabla ^k \S_1 , \p_t \nabla^k \dre \rangle$ as follows:
\begin{equation}\label{korderS1}
  \begin{aligned}
    \langle \nabla ^k \S_1 , \p_t \nabla^k \dre \rangle
    \lesssim & \Big{\{} (1 + | \nabla \dr_0 |_{L^\infty_t H^2_x}) ( | \p_t \dr_0 |^2_{L^\infty_t H^3_x} + | \nabla \dr_0 |^2_{L^\infty_t H^3_x} + | \D(x) |^2_{H^3} ) \\
    & +| \p_{tt} \dr_0 |_{L^\infty_t H^2_x} + |\D(x)|_{H^4} \Big{\}} \big{|} \tfrac{\p_t \dre}{\sqrt{\eps}} \big{|}_{H^2} \, .
  \end{aligned}
\end{equation}

\underline{For the term $ \langle \nabla^k S_2 , \p_t \nabla^k \dre \rangle $}, we can also use the H\"older inequality and Soblev embedding theorems to get
\begin{equation*}
  \begin{aligned}
    & \tfrac{1}{\sqrt{\eps}} \langle \nabla^k \big{(} ( \nabla \dr_0 \cdot \nabla \dre ) \dre \big{)} , \p_t \nabla^k \dre \rangle \\
    \lesssim & \tfrac{1}{\sqrt{\eps}} \sum_{\substack{i+j+e=k \\ e \geqslant 1}} |\langle \nabla^{i+1} \dr_0 \cdot \nabla^{j+1} \dre \nabla^{e} \dre , \p_t \nabla^k \dre \rangle| \\
    & + \tfrac{1}{\sqrt{\eps}} \sum_{\substack{i+j=k}} |\langle \nabla^{i+1} \dr_0 \cdot  \nabla^{j+1} \dre \dre , \p_t \nabla^k \dre \rangle| \\
    \lesssim & | \nabla \dr_0 |_{L^\infty_t H^4_x} | \dre |_{H^2} | \nabla \dre |_{H^2} | \tfrac{\p_t \dre}{\sqrt{\eps}} |_{H^2} \, ,
  \end{aligned}
\end{equation*}
and
\begin{equation*}
  \begin{aligned}
    & \tfrac{1}{\sqrt{\eps}} \langle \nabla^k \big{(} | \nabla \dre |^2 \dr_0 \big{)} , \p_t \nabla^k \dre \rangle \\
    \lesssim & \tfrac{1}{\sqrt{\eps}} \sum_{ \substack{i+j+e=k \\ e \geqslant 1}} |\langle \nabla^{i+1} \dre \nabla^{j+1} \dre \nabla^e \dr_0 , \p_t \nabla^k \dre \rangle| \\
    & + \tfrac{1}{\sqrt{\eps}} \sum_{ \substack{i+j=k}} |\langle \nabla^{i+1} \dre \nabla^{j+1} \dre \dr_0 , \p_t \nabla^k \dre \rangle| \\
    \lesssim & ( 1 + | \nabla \dr_0 |_{L^\infty_t H^3_x}) | \nabla \dre |^2_{H^2} | \tfrac{\p_t \dre}{\sqrt{\eps}} |_{H^2}\, .
  \end{aligned}
\end{equation*}
Summarizing, we obtain
\begin{equation}\label{korderS2}
  \begin{aligned}
    \langle \nabla^k S_2 , \p_t \nabla^k \dre \rangle
    \lesssim ( 1 + | \nabla \dr_0 |_{L^\infty_t H^4_x}) ( | \dre |_{H^2}+ | \nabla \dre |_{H^2}  ) | \nabla \dre |_{H^2} | \tfrac{\p_t \dre}{\sqrt{\eps}} |_{H^2} \, .
  \end{aligned}
\end{equation}

\underline{For the estimate of $ \langle \nabla^k \S_3 , \p_t \nabla^k \dre\rangle $}, we get the estimate of the first term by using again the H\"older inequality and the Sobolev embedding theorems:
\begin{equation*}
  \begin{aligned}
    & \tfrac{1}{\eps} \langle \nabla^k ( |\nabla \dr_0|^2 \dre ) , \p_t \nabla^k \dre \rangle \\
    = & \tfrac{1}{\eps} \sum_{ \substack{i+j+e=k}} \langle \nabla^{i+1} \dr_0    \nabla^{j+1} \dr_0 \nabla^e \dre , \p_t \nabla^k \dre \rangle \\
    \lesssim & | \nabla \dr_0 |^2_{L^\infty_t H^4_x} | \tfrac{\dre}{\sqrt{\eps}} |_{H^2} | \tfrac{\p_t \dre}{\sqrt{\eps}} |_{H^2}\, .
  \end{aligned}
\end{equation*}
Recalling that $ | \dr_0 | = 1 $, one can easily  estimate the second term
\begin{equation*}
  \begin{aligned}
    & \tfrac{1}{\eps} \langle \nabla^k \big{(} ( \nabla \dr_0 \cdot \nabla \dre ) \dr_0 \big{)} , \p_t \nabla^k \dre \rangle \\
    \lesssim & \tfrac{1}{\eps} \sum_{ \substack{i+j+e=k \\ e\geqslant 1}} |\langle \nabla^{i+1} \dr_0 \cdot \nabla^{j+1} \dre \nabla^e \dr_0 , \p_t \nabla^k \dre \rangle| \\
    & + \tfrac{1}{\eps} \sum_{ \substack{i+j=k}} |\langle \nabla^{i+1} \dr_0 \cdot \nabla^{j+1} \dre \dr_0 , \p_t \nabla^k \dre \rangle| \\
    \lesssim & ( 1 + | \nabla \dr_0 |_{L^\infty_t H^3_x} ) | \nabla \dr_0 |_{L^\infty_t H^4_x} | \tfrac{\nabla \dre}{\sqrt{\eps}} |_{H^2} | \tfrac{\p_t \dre}{\sqrt{\eps}} |_{H^2}\, .
  \end{aligned}
\end{equation*}
Thus by the above two estimates we have
\begin{equation}\label{korderS3}
  \begin{aligned}
    \langle \nabla^k \S_3 , \p_t \nabla_k \dre \rangle
    \lesssim ( 1 + | \nabla \dr_0 |_{L^\infty_t H^4_x} ) | \nabla \dr_0 |_{L^\infty_t H^4_x} \big{(} | \tfrac{\dre}{\sqrt{\eps}} |_{H^2} + | \tfrac{\nabla \dre}{\sqrt{\eps}} |_{H^2} \big{)} | \tfrac{\p_t \dre}{\sqrt{\eps}} |_{H^2}\, .
  \end{aligned}
\end{equation}
Then the inequalities \eqref{korderS1}, \eqref{korderS2} and \eqref{korderS3} give the following estimate
\begin{equation}\label{korderS-ptd}
  \begin{aligned}
    \langle \nabla^k \S(\dre) , \p_t \nabla^k \dre \rangle
    \leq C_{k1} \big{(} 1 + | \dre |^2_{H^2} + | \nabla \dre |^2_{H^2} + | \tfrac{\dre}{\sqrt{\eps}} |_{H^2} + | \tfrac{\nabla \dre}{\sqrt{\eps}} |_{H^2} \big{)} | \tfrac{\p_t \dre}{\sqrt{\eps}} |_{H^2}\, ,
  \end{aligned}
\end{equation}
where the positive constant $C_{k1}$ is
$$ C_{k1} = C \Big{\{} \big{(} 1 + | \nabla \dr_0 |_{L^\infty_t H^4_x} \big{)} \big{(} 1+ | \p_t \dr_0 |^2_{L^\infty_t H^3_x} + | \nabla \dr_0 |^2_{L^\infty_t H^4_x} + | \D(x) |^2_{H^3} \big{)} + | \p_{tt} \dr_0 |_{L^\infty_t H^2_x} + |\D(x)|_{H^4} \Big{\}}  $$
for some computable positive constant $ C $.

\bigskip
{ \em (II) Estimates of the regular terms $ \langle \nabla^k \mathcal{R}(\dre) , \p_t \nabla^k \dre \rangle $:}

Finally, we turn to estimating the  regular terms $ \langle \nabla^k \mathcal{R} (\dre) , \p_t \nabla^k \dre \rangle $, which are divided into four parts: $ \langle \nabla^k \mathcal{R}_i , \p_t \nabla^k \dre \rangle (i =1, 2, 3) $.

\underline{For the terms $ \langle \nabla^k \mathcal{R}_1 , \p_t \nabla^k \dre \rangle $ }, by the H\"older inequality and the Sobolev embedding theorems we have:
\begin{equation*}
  \begin{aligned}
    & \sqrt{\eps} \langle \nabla^k \big{(} | \p_t \dr_0 + \D(x) \exp(-\tfrac{t}{\eps}) |^2  \D(x) \exp(-\tfrac{t}{\eps}) \big{)} , \p_t \nabla^k \dre \rangle \\
    \lesssim & \sqrt{\eps} \sum_{\substack{i+j+e=k}} |\langle \nabla^i ( \p_t \dr_0 + \D(x) ) \nabla^j ( \p_t \dr_0 + \D(x)) \nabla^e \D(x) , \p_t \nabla^k \dre \rangle| \\
    \lesssim & \sqrt{\eps} |\D(x)|_{H^3} ( |\D(x)|^2_{H^2} + |\p_t \dr_0|^2_{L^\infty_t H^2_x}) |\p_t \dre|_{H^2}\, ,
  \end{aligned}
\end{equation*}
and
\begin{equation*}
  \begin{aligned}
    & \sqrt{\eps} \langle \nabla^k \big{(} |\nabla \D(x)|^2 \exp(-\tfrac{2t}{\eps}) \dr_0 \big{)} , \p_t \nabla^k \dre \rangle \\
    \lesssim & \sqrt{\eps} \sum_{\substack{i+j+e=k \\ e \geqslant 1}} |\langle \nabla^{i+1} \D(x) \nabla^{j+1} \D(x) \nabla^e \dr_0 , \p_t \nabla^k \dr_0 \rangle| \\
    & + \sqrt{\eps} \sum_{ \substack{i+j=k} } |\langle \nabla^{i+1} \D(x) \nabla^{j+1} \D(x) \dr_0 , \p_t \nabla^k \dre \rangle| \\
    \lesssim & \sqrt{\eps} ( 1 + | \nabla \dr_0 |_{L^\infty_t H^2_x} )| \D(x) |^2_{H^3} | \p_t \dre |_{H^2}\, .
  \end{aligned}
\end{equation*}
We can estimate the following two terms in a similar way, hence we get the following inequalities:
\begin{equation*}
  \begin{aligned}
    & \eps^{\frac{3}{2}} \langle \nabla^k \big{(} | \nabla \D(x) |^2 \exp(-\tfrac{3t}{\eps}) \D(x) \big{)} , \p_t \nabla^k \dre \rangle \\
    \lesssim & \eps^{\frac{3}{2}} \sum_{ \substack{i+j+e=k}} |\langle \nabla^{i+1} \D(x) \nabla^{j+1} \D(x) \nabla^e \D(x) , \p_t \nabla^k \dre \rangle|\\
    \lesssim & \eps^{\frac{3}{2}} |\D(x)|^3_{H^3} |\p_t \dre|_{H^2}
  \end{aligned}
\end{equation*}
and
\begin{equation*}
  \begin{aligned}
    & \sqrt{\eps} \langle \nabla^k \big{(} \nabla \dr_0 \cdot \nabla \D(x) ) \exp(-\tfrac{2t}{\eps}) \D(x) \big{)} , \p_t \nabla^k \dre \rangle \\
    \lesssim & \sqrt{\eps} \sum_{ \substack{i+j+e=k}} |\langle \nabla^{i+1} \dr_0 \nabla^{j+1} \D(x) \nabla^e \D(x) , \p_t \nabla^k \dre \rangle|\\
    \lesssim & \sqrt{\eps} | \nabla \dr_0 |_{L^\infty_t H^4_x} | \D(x) |^2_{H^3} | \p_t \dre |_{H^2} \, .
  \end{aligned}
\end{equation*}
 So we have the estimate of $ \langle \nabla^k \mathcal{R}_1 , \p_t \nabla^k \dre \rangle $ as follows:
\begin{equation}\label{korderR1}
  \begin{aligned}
    \langle \nabla^k \mathcal{R}_1 , \p_t \nabla^k \dre \rangle
    \lesssim& \sqrt{\eps} ( 1 + |\nabla \dr_0|_{L^\infty_t H^4_x} + |\D(x)|_{H^3} )\\
     & \times ( | \p_t \dr_0 |^2_{L^\infty_t H^2_x}+ | \D(x)|^2_{H^3} ) | \p_t \dre |_{H^2}\, .
  \end{aligned}
\end{equation}

\underline{For the terms $ \langle \nabla^k \mathcal{R}_2 , \p_t \nabla^k \dre \rangle $ }, by using yet again the H\"older inequality and the Sobolev embedding theorems, we have

\begin{equation*}
  \begin{aligned}
    & \eps \langle \nabla^k \big{(} \p_t \dr_0+\D(x) \exp(-\tfrac{t}{\eps})\big{)} \p_t \dre D(x)\exp(-\tfrac{t}{\eps}) , \p_t \nabla^k \dre \rangle \\
    \lesssim & \eps \sum_{ \substack{i+j+e=k \\ e \geqslant 1} } |\langle  \big{(} \nabla^i \p_t \dr_0 +\nabla^i \D(x) \exp(-\tfrac{t}{\eps})\big{)}\nabla^j \p_t \dre \nabla^e D(x)\exp(-\tfrac{t}{\eps}), \p_t \nabla^k \dre \rangle| \\
    &+\eps \sum_{ \substack{i+j=k } } |\langle \big{(}  \nabla^i \p_t \dr_0 +\nabla^i \D(x) \exp(-\tfrac{t}{\eps})\big{)}\nabla^j \p_t \dre  D(x)\exp(-\tfrac{t}{\eps}), \p_t \nabla^k \dre \rangle| \\
    \lesssim & \eps (| \p_t \dr_0 |_{L^\infty_t H^4_x} +| \D(x) |_{H^4})  | \D(x) |_{H^4}|\p_t \dre |^2_{H^2}\, ,
  \end{aligned}
\end{equation*}
Similarly as for estimating the term $ \eps \langle \nabla^k \big{(} \p_t \dr_0+\D(x) \exp(-\tfrac{t}{\eps})\big{)} \p_t \dre D(x)\exp(-\tfrac{t}{\eps}) , \p_t \nabla^k \dre \rangle$ we can easily obtain the following estimates:
\begin{equation*}
  \begin{aligned}
    & \langle \nabla^k \big{(} | \p_t \dr_0 + \D(x) \exp(-\tfrac{t}{\eps}) |^2 \dre \big{)} , \p_t \nabla^k \dre \rangle
    \lesssim ( | \p_t \dr_0 |^2_{L^\infty_t H^4_x} + | \D(x) |^2_{H^4} ) | \dre |_{H^2} | \p_t \dre |_{H^2}\, ,  \\
    & \langle \nabla^k \big{(} \nabla \dr_0 \cdot \nabla \dre \D(x) \exp(-\tfrac{t}{\eps}) \big{)} , \p_t \nabla^k \dre \rangle
    \lesssim ( | \nabla \dr_0 |^2_{L^\infty_t H^4_x} + | \D(x) |^2_{H^4} ) | \nabla \dre |_{H^2} | \p_t \dre |_{H^2} \, , \\
    & \langle \nabla^k \big{(} \nabla \dr_0 \cdot \nabla \D(x) \exp(-\tfrac{t}{\eps}) \dre \big{)} , \p_t \nabla^k \dre \rangle
    \lesssim ( | \nabla \dr_0 |^2_{L^\infty_t H^4_x} + | \D(x) |^2_{H^4} ) | \dre |_{H^2} | \p_t \dre |_{H^2} \, .
  \end{aligned}
\end{equation*}
Observing the structure of the terms $ \eps \langle \nabla^k \big{(} |\nabla \D(x)|^2 \exp(-\tfrac{2t}{\eps}) \dre \big{)} , \p_t \nabla^k \dre \rangle $ and $ \eps \langle \nabla^k \big{(} \nabla \D(x) \cdot \nabla \dre  \exp(-\tfrac{2t}{\eps}) \D(x) \big{)} , \p_t \nabla^k \dre \rangle $ one can similarly estimate the following terms:
\begin{equation*}
  \begin{aligned}
    & \eps \langle \nabla^k \big{(} |\nabla \D(x)|^2 \exp(-\tfrac{2t}{\eps}) \dre \big{)} , \p_t \nabla^k \dre \rangle
      \lesssim \eps | \D(x) |^2_{H^5} | \dre |_{H^2} | \p_t \dre |_{H^2}\, , \\
    & \eps \langle \nabla^k \big{(} \nabla \D(x) \cdot \nabla \dre  \exp(-\tfrac{2t}{\eps}) \D(x) \big{)} , \p_t \nabla^k \dre \rangle
    \lesssim \eps | \D(x) |^2_{H^5} | \nabla \dre |_{H^2} | \p_t \dre |_{H^2}\, .
  \end{aligned}
\end{equation*}

Furthermore we get:
\begin{equation*}
  \begin{aligned}
    & \langle \nabla^k [ \big{(} \p_t \dr_0 + \D(x) \exp(-\tfrac{t}{\eps}) \big{)}\cdot \p_t \dre \dr_0 ] , \p_t \nabla^k \dre \rangle \\
    \lesssim & \sum_{ \substack{i+j+e=k \\ e \geqslant 1} } |\langle \nabla^i ( \p_t \dr_0 + \D(x) ) \nabla^j \p_t \dre \nabla^e \dr_0 , \p_t \nabla^k \dre \rangle| \\
    & + \sum_{ \substack{i+j=k} } |\langle \nabla^i ( \p_t \dr_0 + \D(x) ) \nabla^j \p_t \dre \dr_0 , \p_t \nabla^k \dre \rangle| \\
    \lesssim & ( 1 + | \nabla \dr_0 |_{L^\infty_t H^3_x}) ( | \p_t \dr_0 |_{L^\infty_t H^4_x} + | \D(x) |_{H^4} ) | \p_t \dre |^2_{H^2}
  \end{aligned}
\end{equation*}
 Similarly as for estimating $ \langle \nabla^k [ \big{(} \p_t \dr_0 + \D(x) \exp(-\tfrac{t}{\eps}) \big{)}\cdot \p_t \dre \dr_0 ] , \p_t \nabla^k \dre \rangle $, one can also get:
\begin{equation*}
  \begin{aligned}
    & \langle \nabla^k \big{(} \nabla \D(x) \cdot \nabla \dre \exp(-\tfrac{t}{\eps}) \dr_0 \big{)} , \p_t \nabla^k \dre \rangle
    \lesssim & ( 1 + | \nabla \dr_0 |_{L^\infty_t H^3_x}) | \D(x) |_{H^5} | \nabla \dre |_{H^2} | \p_t \dre |_{H^2} \, .
  \end{aligned}
\end{equation*}
Thus we have the following estimate of $  \langle \nabla^k \mathcal{R}_2 , \p_t \nabla^k \dre \rangle $:
\begin{equation}\label{korderR2}
  \begin{aligned}
    \langle \nabla^k \mathcal{R}_2 , \p_t \nabla^k \dre \rangle
    \lesssim & ( | \p_t \dr_0 |_{L^\infty_t H^4_x} + | \D(x) |_{H^5} + | \p_t \dr_0 |^2_{L^\infty_t H^4_x} + | \D(x) |^2_{H^5} + | \nabla \dr_0 |^2_{L^\infty_t H^4_x} )\\
    & \times ( | \dre |_{H^2} + | \nabla \dre |_{H^2} + | \p_t \dre |_{H^2} )| \p_t \dre |_{H^2} \, .
  \end{aligned}
\end{equation}

\bigskip
\underline{For the terms $ \langle \nabla^k \mathcal{R}_3 , \p_t \nabla^k \dre \rangle, $ } one can use an estimate  similar to the one for the term $ \eps \langle \nabla^k \big{(} \p_t \dr_0 \cdot \D(x) \exp(-\tfrac{t}{\eps}) \p_t \dre \big{)} , \p_t \nabla^k \dre \rangle $ to get:
\begin{equation*}
  \begin{aligned}
    & \eps^{\frac{3}{2}} \langle \nabla^k \big{(} \D(x) \exp(-\tfrac{t}{\eps}) | \p_t \dre |^2 \big{)} , \p_t \nabla^k \dre \rangle
    \lesssim \eps^{\frac{3}{2}} | \D(x) |_{H^4} | \p_t \dre |^3_{H^2}\, , \\
    & \sqrt{\eps} \langle \nabla^k \big{(} \nabla \D(x) \cdot \nabla \dre \exp(-\tfrac{t}{\eps}) \dre \big{)} , \p_t \nabla^k \dre \rangle
    \lesssim \sqrt{\eps} | \D(x) |_{H^5} ( | \dre |^2_{H^2} +| \nabla \dre |^2_{H^2} ) | \p_t \dre |_{H^2}\, , \\
    & \sqrt{\eps} \langle \nabla^k \big{(} | \nabla \dre |^2 \D(x) \exp(-\tfrac{t}{\eps}) \big{)} , \p_t \nabla^k \dre \rangle
    \lesssim \sqrt{\eps} | \D(x) |_{H^4} | \nabla \dre |^2_{H^2} | \p_t \dre |_{H^2} \, .
  \end{aligned}
\end{equation*}
Reasoning analogously as in estimating $ \sqrt{\eps} \langle \nabla^k [ \big{(} \p_t \dr_0 + \D(x) \exp(-\tfrac{t}{\eps}) \big{)} \cdot \p_t \dre \dre ] , \p_t \nabla^k \dre \rangle $ we have
\begin{equation*}
  \begin{aligned}
    & \sqrt{\eps} \langle \nabla^k [ \big{(} \p_t \dr_0 + \D(x) \exp(-\tfrac{t}{\eps}) \big{)} \cdot \p_t \dre \dre ] , \p_t \nabla^k \dre \rangle \\
    \lesssim & \sqrt{\eps} \sum_{ \substack{i+j+e=k \\ e \geqslant 1} } |\langle \nabla^i ( \p_t \dr_0 + \D(x) ) \cdot \nabla^j \p_t \dre \nabla^e \dre , \p_t \nabla^k \dre \rangle| \\
    & + \sqrt{\eps} \sum_{ \substack{i+j=k} } |\langle \nabla^i ( \p_t \dr_0 + \D(x) ) \cdot \nabla^j \p_t \dre \dre , \p_t \nabla^k \dre \rangle| \\
    \lesssim & \sqrt{\eps} ( | \p_t \dr_0 |_{L^\infty_t H^4_x} + | \D(x) |_{H^4} ) ( | \dre |_{H^2} + | \nabla \dre |_{H^2} )| \p_t \dre |^2_{H^2}
  \end{aligned}
\end{equation*}
and furthermore, using  that $ | \dr_0 | = 1 $, it is easy to obtain
\begin{equation*}
  \begin{aligned}
    & \sqrt{\eps} \langle \nabla^k ( | \p_t \dre |^2 \dr_0 ), \p_t \nabla^k \dre \rangle
    \lesssim \sqrt{\eps} ( 1 + | \nabla \dr_0 |_{L^\infty_t H^3_x} )| \p_t \dre |^3_{H^2}\, .
  \end{aligned}
\end{equation*}
Summarizing, we get the estimate of $ \langle \nabla^k \mathcal{R}_3 , \p_t \nabla^k \dre \rangle $ as follows:
\begin{equation}\label{korderR3}
  \begin{aligned}
    \langle \nabla^k \mathcal{R}_3 , \p_t \nabla^k \dre \rangle
    \lesssim & \sqrt{\eps} ( 1 + | \p_t \dr_0 |_{L^\infty_t H^4_x} + | \nabla \dr_0 |_{L^\infty_t H^3_x} + | \D(x) |_{H^5} ) \\
    & \times ({\color{red} 1+} | \p_t \dre |^2_{H^2} + | \dre |^2_{H^2} + | \nabla \dre |^2_{H^2} ) | \p_t \dre |_{H^2} \, .
  \end{aligned}
\end{equation}

\bigskip
\underline{For the terms $ \langle \nabla^k \mathcal{R}_4 , \p_t \nabla^k \dre \rangle $ }, we get, by the H\"older inequality and Sobolev embedding theorems:
  \begin{align*}
    & \eps \langle \nabla^k ( | \p_t \dre |^2 \dre ) , \p_t \nabla^k \dre \rangle \\
    = & \eps \sum_{ \substack{i+j+e=k \\ e\geqslant 1} } \langle \nabla^i \p_t \dre \nabla^j \p_t \dre \nabla^e \dre , \p_t \nabla^k \dre \rangle \\
    & + \eps \sum_{ \substack{i+j=k} } \langle \nabla^i \p_t \dre \nabla^j \p_t \dre \dre , \p_t \nabla^k \dre \rangle \\
    \lesssim & \eps ( | \dre |_{H^2} + | \nabla \dre |_{H^2} )| \p_t \dre |^3_{H^2}
  \end{align*}
and then, similarly:

\begin{equation*}
  \begin{aligned}
    & \langle \nabla^k ( | \nabla \dre |^2 \dre ) , \p_t \nabla^k \dre \rangle
    \lesssim  ( | \dre |_{H^2} + | \nabla \dre |_{H^2} ) | \nabla \dre |^2_{H^2} | \p_t \dre |_{H^2}\, .
  \end{aligned}
\end{equation*}
So we obtain the estimate of $ \langle \nabla^k \mathcal{R}_4 , \p_t \nabla^k \dre \rangle$ as follows:
\begin{equation}\label{korderR4}
  \begin{aligned}
    \langle \nabla^k \mathcal{R}_4 , \p_t \nabla^k \dre \rangle
    \lesssim ( | \dre |_{H^2} + | \nabla \dre |_{H^2} ) ( | \nabla \dre |^2_{H^2} + \eps | \p_t \dre |^2_{H^2} ) | \p_t \dre |_{H^2}\, .
  \end{aligned}
\end{equation}

Then the inequalities \eqref{korderR1}, \eqref{korderR2}, \eqref{korderR3} and \eqref{korderR4} give the estimate of the regular terms $ \langle \nabla^k \mathcal{R} ( \dre ) , \p_t \nabla^k \dre \rangle $ as follows:
\begin{equation}\label{korderR-ptd}
  \begin{aligned}
    \langle \nabla^k \mathcal{R} (\dre) , \p_t \nabla^k \dre \rangle
    \leq & C_{k2} \big{[} ( | \dre |_{H^2} + | \nabla \dre |_{H^2} ) ( 1 + | \nabla \dre |^2_{H^2} + \eps | \p_t \dre |^2_{H^2}) \\
    & + \sqrt{\eps} ( 1 +  | \p_t \dre |^2_{H^2} + | \dre |^2_{H^2} + | \nabla \dre |^2_{H^2} ) + | \p_t \dre |_{H^2} \big{]} | \p_t \dre |_{H^2}\, ,
  \end{aligned}
\end{equation}
where the constant $C_{k2}$ is
$$ C_{k2} = C \big{(} 1 + |\p_t \dr_0|_{L^\infty_t H^4_x} + |\nabla \dr_0|_{L^\infty_t H^4_x} + |\D(x)|_{H^5} \big{)} \big{(} 1+ |\p_t \dr_0|^2_{L^\infty_t H^4_x} + |\D(x)|^2_{H^5} + |\nabla \dr_0|_{L^\infty_t H^4_x}\big{)} > 0 ,$$
and $ C $ is a computable positive constant.

Therefore, by substituting the inequalities \eqref{korderS-ptd} and \eqref{korderR-ptd} into \eqref{korderInit-ptd} one has
\begin{equation}\label{korder-ptd}
  \begin{aligned}
    & \tfrac{1}{2} \tfrac{\d}{\d t} \big{(} | \p_t \nabla^k \dr^\eps_R |^2_{L^2} + \tfrac{1}{\eps} | \nabla^{k+1} \dr^\eps_R |^2_{L^2} \big{)} + \tfrac{1}{\eps} | \p_t \nabla^k \dr^\eps_R |^2_{L^2}\\
     \leq & C_{k} \Big{\{} \big{(} 1 + | \dre |^2_{H^2} + | \nabla \dre |^2_{H^2} + | \tfrac{\dre}{\sqrt{\eps}} |_{H^2} + | \tfrac{\nabla \dre}{\sqrt{\eps}} |_{H^2} \big{)} | \tfrac{\p_t \dre}{\sqrt{\eps}} |_{H^2} \\
     & + \big{[} ( | \dre |_{H^2} + | \nabla \dre |_{H^2} ) ( 1 + | \nabla \dre |^2_{H^2} + \eps | \p_t \dre |^2_{H^2}) \\
    & + \sqrt{\eps} ( 1 +  | \p_t \dre |^2_{H^2} + | \dre |^2_{H^2} + | \nabla \dre |^2_{H^2} ) + | \p_t \dre |_{H^2} \big{]} | \p_t \dre |_{H^2} \Big{\}} \, ,
  \end{aligned}
\end{equation}
where the positive constant $C_k$ is
$$ C_k = C \big{(} 1 + |\partial_{tt} \dr_0|_{L^\infty(0,T;H^2)} + |\partial_t \dr_0|^3_{L^\infty(0,T;H^4)} + |\nabla \dr_0|^3_{L^\infty(0,T;H^4)} + |\D(x)|^3_{H^5} \big{)} > 0 $$
and $C > 0$ is a computable constant.

\bigskip
{ \em (III) For the estimate of $ | \nabla^k \dre |_{L^2} $ $ (k = 1, 2) $:}

 Applying $ \nabla^k ( k = 1, 2 ) $ to the remainder equation \eqref{Remainder-Eq}, multiplying by $ \nabla^k \dre $, integrating over $ \mathbb{R}^3 $ and by parts, we have
\begin{equation}\label{korderInit-d}
  \begin{aligned}
    \tfrac{1}{2} \tfrac{\d }{\d t} & \Big{[} | \nabla^k \p_t \dre + \nabla^k \dre |^2_{L^2} + \big{(} \tfrac{1}{\eps} - 1 \big{)} | \nabla^k \dre |^2_{L^2} - | \p_t \nabla^k \dre |^2_{L^2} \Big{]} - | \p_t \nabla^k \dre |^2_{L^2} + \tfrac{1}{\eps} | \nabla^{k+1} \dre |^2_{L^2} \\
   & = \langle \nabla^k \S (\dre) , \nabla^k \dre \rangle + \langle \nabla^k \mathcal{R} (\dre) , \nabla^k \dre \rangle\, .
  \end{aligned}
\end{equation}
Similarly as in the estimates of the terms $ \langle \nabla^k \S(\dre) , \p_t \nabla^k \dre \rangle $ and $ \langle \nabla^k \mathcal{R} (\dre) , \p_t \nabla^k \dre \rangle $ in the inequalities \eqref{korderS-ptd} and \eqref{korderR-ptd}, respectively,
we can analogously estimate the terms $ \langle \nabla^k \S(\dre) , \nabla^k \dre \rangle $ and $ \langle \nabla^k \mathcal{R} (\dre) , \nabla^k \dre \rangle $ as follows:
\begin{equation}\label{korderS-d}
  \begin{aligned}
    \langle \nabla^k \S(\dre) , \p_t \nabla^k \dre \rangle
    \leq C_{k1} \big{(} 1 + | \dre |^2_{H^2} + | \nabla \dre |^2_{H^2} + | \tfrac{\dre}{\sqrt{\eps}} |_{H^2} + | \tfrac{\nabla \dre}{\sqrt{\eps}} |_{H^2} \big{)} | \tfrac{\dre}{\sqrt{\eps}} |_{H^2}\, ,
  \end{aligned}
\end{equation}
and
\begin{equation}\label{korderR-d}
  \begin{aligned}
    \langle \nabla^k \mathcal{R} (\dre) , \p_t \nabla^k \dre \rangle
    \leq & C_{k2} \big{[} ( | \dre |_{H^2} + | \nabla \dre |_{H^2} ) ( 1 + | \nabla \dre |^2_{H^2} + \eps | \p_t \dre |^2_{H^2}) \\
    & + \sqrt{\eps} ( 1 +  | \p_t \dre |^2_{H^2} + | \dre |^2_{H^2} + | \nabla \dre |^2_{H^2} ) + | \p_t \dre |_{H^2} \big{]} |\dre |_{H^2}\, .
  \end{aligned}
\end{equation}
By plugging the inequalities \eqref{korderS-d} and \eqref{korderR-d} into the equality \eqref{korderInit-d}, we  get the following estimate:
\begin{equation}\label{korder-d}
  \begin{aligned}
     \tfrac{1}{2} \tfrac{\d }{\d t} & \Big{[} | \nabla^k \p_t \dre + \nabla^k \dre |^2_{L^2} + \big{(} \tfrac{1}{\eps} - 1 \big{)} | \nabla^k \dre |^2_{L^2} \\
    &{\color{red} - | \p_t \nabla^k \dre |^2_{L^2} }\Big{]} - | \p_t \nabla^k \dre |^2_{L^2} + \tfrac{1}{\eps} | \nabla^{k+1} \dre |^2_{L^2} \\
   & \leq C_{k} \Big{\{} \big{(} 1 + | \dre |^2_{H^2} + | \nabla \dre |^2_{H^2} + | \tfrac{\dre}{\sqrt{\eps}} |_{H^2} + | \tfrac{\nabla \dre}{\sqrt{\eps}} |_{H^2} \big{)} | \tfrac{\dre}{\sqrt{\eps}} |_{H^2} \\
   & + \big{[} ( | \dre |_{H^2} + | \nabla \dre |_{H^2} ) ( 1 + | \nabla \dre |^2_{H^2} + \eps | \p_t \dre |^2_{H^2}) \\
   & + \sqrt{\eps} ( 1 +  | \p_t \dre |^2_{H^2} + | \dre |^2_{H^2} + | \nabla \dre |^2_{H^2} ) + | \p_t \dre |_{H^2} \big{]} | \dre |_{H^2} \Big{\}} \, .
  \end{aligned}
\end{equation}

Multiplying  the inequality \eqref{korder-d} by $ \tfrac{1}{2} $ and adding it to the inequality \eqref{korder-ptd}, we obtain the higher order estimate:
\begin{equation}\label{korderEst}
  \begin{aligned}
    & \tfrac{1}{4} \tfrac{\d }{\d t} \Big{[} | \p_t \nabla^k \dre |^2_{L^2} + \big{(} \tfrac{1}{\eps} - 1 \big{)} | \nabla^k \dre |^2_{L^2} +\tfrac{2}{\eps}| \nabla^{k+1} \dre |^2_{L^2} + | \nabla^k \p_t \dre + \nabla^k \dre |^2_{L^2} \Big{]} \\
    & \quad + \big{(} \tfrac{1}{\eps} - \tfrac{1}{2} \big{)} | \p_t \nabla^k \dre |^2_{L^2} + \tfrac{1}{2\eps} | \nabla^{k+1} \dre |^2_{L^2} \\
   & \leq \frac 32  C_{k} \Big{\{} \big{(} 1 + | \dre |^2_{H^2} + | \nabla \dre |^2_{H^2} + | \tfrac{\dre}{\sqrt{\eps}} |_{H^2} + | \tfrac{\nabla \dre}{\sqrt{\eps}} |_{H^2} \big{)} \big{(} | \tfrac{\dre}{\sqrt{\eps}} |_{H^2} + | \tfrac{\p_t \dre}{\sqrt{\eps}} |_{H^2} \big{)}\\
   & \quad + \big{[} ( | \dre |_{H^2} + | \nabla \dre |_{H^2} ) ( 1 + | \nabla \dre |^2_{H^2} + \eps | \p_t \dre |^2_{H^2})  + | \p_t \dre |_{H^2} \\
   & \quad + \sqrt{\eps} ( 1 +  | \p_t \dre |^2_{H^2} + | \dre |^2_{H^2} + | \nabla \dre |^2_{H^2} ) \big{]} \big{(} | \dre |_{H^2} + | \p_t \dre |_{H^2} \big{)} \Big{\}} \,.
  \end{aligned}
\end{equation}

Therefore, combining the $L^2$-estimate \eqref{0orderEst} and the $k$th-order estimate \eqref{korderEst} for $k=1, 2$, we obtain:
\begin{equation}\label{EnergyEst}
  \begin{aligned}
    & \tfrac{1}{4} \tfrac{\d }{\d t} \Big{[} | \p_t \dre |^2_{H^2} + \big{(} \tfrac{1}{\eps} - 1 \big{)} | \dre |^2_{H^2} + \tfrac{2}{\eps}| \nabla \dre |^2_{H^2} + | \p_t \dre +  \dre |^2_{H^2} \Big{]} \\
    & \quad + \big{(} \tfrac{1}{\eps} - \tfrac{1}{2} \big{)} | \p_t \dre |^2_{H^2} + \tfrac{1}{2\eps} | \nabla \dre |^2_{H^2} \\
   & \leq  \tilde C_{k} \Big{\{} \big{(} 1 + | \dre |^2_{H^2} + | \nabla \dre |^2_{H^2} + | \tfrac{\dre}{\sqrt{\eps}} |_{H^2} + | \tfrac{\nabla \dre}{\sqrt{\eps}} |_{H^2} \big{)} \big{(} | \tfrac{\dre}{\sqrt{\eps}} |_{H^2} + | \tfrac{\p_t \dre}{\sqrt{\eps}} |_{H^2} \big{)}\\
   & \quad + \big{[} ( | \dre |_{H^2} + | \nabla \dre |_{H^2} ) ( 1 + | \nabla \dre |^2_{H^2} + \eps | \p_t \dre |^2_{H^2})  + | \p_t \dre |_{H^2} \\
   & \quad + \sqrt{\eps} ( 1 +  | \p_t \dre |^2_{H^2} + | \dre |^2_{H^2} + | \nabla \dre |^2_{H^2} ) \big{]} \big{(} | \dre |_{H^2} + | \p_t \dre |_{H^2} \big{)} \Big{\}} \,,
  \end{aligned}
\end{equation}
where the positive constant $\tilde C_k$ is
$$ \tilde C_k = C \big{(} 1 + |\partial_{tt} \dr_0|_{L^\infty(0,T;H^2)} + |\partial_t \dr_0|^3_{L^\infty(0,T;H^4)} + |\nabla \dr_0|^3_{L^\infty(0,T;H^4)} + |\D(x)|^3_{H^5} \big{)} > 0 $$
and $C > 0$ is a computable constant. Then, by the definition of the energy functionals $ E_{\eps}(t) $ and $ F_{\eps}(t) $, and the condition $0 < \eps < \frac{1}{2}$, the $ H^2$-estimate \eqref{EnergyEst} implies that
\begin{equation*}
  \frac{\d}{\d t} E_{\eps}(t) + 4 F_{\eps}(t)
  \leq C' \Big{\{} E^{\frac{1}{2}}_{\eps}(t) +  E_{\eps}(t) + \eps^{\frac{1}{2}}E^{\frac{3}{2}}_{\eps}(t) + \eps^{\frac{3}{2}} E^2_{\eps}(t) + [1 + E^{\frac{1}{2}}_{\eps}(t) + \eps E_{\eps}(t) ] F^{\frac{1}{2}}_{\eps}(t) \Big{\}} \,,
\end{equation*}
where $C'  > 0$, which immediately implies the claimed inequality \eqref{Energy-Bounds} by using Young's inequality. Consequently, the proof of Lemma \ref{Lm-Unif-Bounds} is completed.
\end{proof}

\section{The Proof of Theorem \ref{Main-Thm}}\label{Sct-Proof}

In this section we will provide the proof of Theorem \ref{Main-Thm},  by using the uniform energy bounds \eqref{Energy-Bounds} in Section \ref{Sct-Unif-Bounds}. Before doing this, we note that for any fixed inertia constant $\eps > 0$ the  well-posedness of the remainder system \eqref{Remainder-Eq}-\eqref{IC-Remainder} can be stated as follows:

\begin{proposition}\label{Prop-WM}
 Given $\dr^{in}:\mathbb{R}^3\to\mathbb{S}^2$ and $\tilde{\dr}^{in}:\mathbb{R}^3\to\mathbb{R}^3$ satisfying $\nabla \dr^{in} \in H^6\, , \ \tilde{\dr}^{in} \in H^4$  with $\dr^{in}\cdot \tilde{\dr}^{in}\equiv 0$, we define $ \D(x) \equiv \tilde{\dr}^{in} (x) - \Delta \dr^{in}(x) - |\nabla \dr^{in}(x)|^2 \dr^{in}(x) $ and denote $M = |\D|^2_{H^2} + 2 |\nabla \D|^2_{H^2} < \infty$.

  Then, for any fixed $\eps \in ( 0, \tfrac{1}{2} )$, there exists a time $T^\eps = \min \{ T , \tfrac{1}{C} \ln \big( \tfrac{1 + \eps M}{\eps ( 1 + M )} \big) \} > 0$, where $T, \ C>0$ are provided in Proposition \ref{Prop-HF}, Lemma \ref{Lm-Unif-Bounds}, respectively, such that the remainder equation \eqref{Remainder-Eq} with the initial conditions \eqref{IC-Remainder} admits a unique solution $\dr^\eps_R \in C ([0,T^\eps); {H}^{3})$ and $\partial_t \dr^\eps_R \in C([0,T^\eps); H^2)$. Moreover, the solution $\dr^\eps_R$ satisfies the inequality
  \begin{equation}\label{Bnds-Energ}
     |\partial_t \dr^\eps_R (t)|^2_{ H^2} + \tfrac{1}{\eps} | \dr^\eps_R (t)|^2_{H^3}  \leq \tfrac{ 2 M e^{C t} }{ 1 + \eps M - \eps ( 1 + M ) e^{C t} }
  \end{equation}
  for all $t \in [ 0 , T^\eps )$.
\end{proposition}

\begin{proof}
We employ a  mollifier argument to prove this proposition. For any fixed $\eps > 0$ we can directly construct a system approximating \eqref{Remainder-Eq}-\eqref{IC-Remainder} as follows:
\begin{equation}\label{Appro-WM}
  \left\{
    \begin{array}{c}
      \eps \partial_t \w^\eps_\eta = - \mathcal{J}_\eta \w^\eps_\eta + \mathcal{J}_\eta \Delta \dr^\eps_{R,\eta} + \eps \mathcal{J}_\eta \mathcal{S} ( \mathcal{J}_\eta \dr^\eps_{R,\eta} ) + \eps \mathcal{J}_\eta \mathcal{R} ( \mathcal{J}_\eta \dr^\eps_{R,\eta} )  \,, \\
      \partial_t \dr^\eps_{R,\eta} = \w^\eps_\eta \,, \\
      \dr^\eps_{R,\eta} \big{|}_{t=0} = \sqrt{\eps} \mathcal{J}_\eta \D (x)\,,\ \w^\eps_\eta \big{|}_{t=0} = 0\, ,
    \end{array}
  \right.
\end{equation}
where the mollifier operator $\mathcal{J}_\eta$ is defined as
$$ \mathcal{J}_\eta f = \mathcal{F}^{-1} \big{(}  \mathbf{1}_{|\xi| \leq \frac{1}{\eta}} \mathcal{F} (f) (\xi) \big{)} \, , $$
where the symbol $\mathcal{F}$ denotes  the standard Fourier transform operator and $\mathcal{F}^{-1}$ is the inverse Fourier transform operator. By ODE theory in Hilbert spaces one can prove the existence and uniqueness of the approximate system \eqref{Appro-WM} on the maximal time interval $[0 , T_\eta^\eps )$. Then by the fact $\mathcal{J}_\eta^2 = \mathcal{J}_\eta$ and the uniqueness of \eqref{Appro-WM} we know that $\mathcal{J}_\eta \dr^\eps_{R,\eta} = \dr^\eps_{R,\eta}$ and $\mathcal{J}_\eta \w^\eps_\eta = \w^\eps_\eta$. Thus by the analogous energy estimate shown in Lemma \ref{Lm-Unif-Bounds} applied to the approximate system \eqref{Appro-WM}, one can obtain the following energy inequality for $\dr_{R, \eta}^\eps$ and $\w_\eta^\eps$
\begin{equation}\label{Unf-Bd-eps-eta}
  \frac{\d}{\d t} E_{\eps, \eta} (t) + 3 F_{\eps, \eta} (t) \leq C \big{[} 1 + E_{\eps , \eta} (t) \big] \big[ 1 + \eps E_{\eps, \eta} (t) \big]
\end{equation}
for all $t \in [ 0 , T^\eps_\eta )$, where the positive constant $C > 0$ is independent of $\eps $ and $\eta$, and the energy functionals $E_{\eps , \eta} (t)$, $F_{\eps,\eta} (t)$ are of the same forms as $E_\eps(t)$, $F_\eps(t)$ defined in Section \ref{Sct-Unif-Bounds} (replacing $\dr_R^\eps$ by $\dr^\eps_{R, \eta}$), respectively.

Since $\dr_{R,\eta}^\eps$ satisfies the initial conditions $\partial_t \dr_{R,\eta}^\eps (0, x) = 0$ and $\dr_{R,\eta}^\eps (0,x) =  \sqrt{\eps} \mathcal{J}_\eta \D(x)$, we know that for $\eps \in (0,\frac{1}{2})$
\begin{equation}\label{IV-Bnds}
  \begin{aligned}
    E_{\eps,\eta}(0) =& | \p_t \dr^\eps_{R,\eta} (0 ,\cdot) |^2_{H^2} + ( \tfrac{1}{\eps} - 1 ) | \dr^\eps_{R,\eta} (0, \cdot) |^2_{H^2} \\
    &+ \tfrac{2}{\eps}| \nabla \dr^\eps_{R,\eta} (0,\cdot) |^2_{H^2} + | \p_t \dr^\eps_{R,\eta} (0,\cdot) +  \dr^\eps_{R,\eta} (0,\cdot) |^2_{H^2} \\
    =& ( \tfrac{1}{\eps} - 1 ) |  \sqrt{\eps} \mathcal{J}_\eta \D |^2_{H^2} + \tfrac{2}{\eps}|\sqrt{\eps} \nabla \mathcal{J}_\eta \D |^2_{H^2} + |  \sqrt{\eps} \mathcal{J}_\eta \D |^2_{H^2} \\
    \leq & (1 - \eps) |\D|^2_{H^2} + 2 |\nabla \D|^2_{H^2} + \eps |\D|^2_{H^2} \\
    = &  |\D|^2_{H^2} + 2 |\nabla \D|^2_{H^2} = M < \infty \,.
  \end{aligned}
\end{equation}
Then, one can solve the ODE inequality \eqref{Unf-Bd-eps-eta} with the initial condition \eqref{IV-Bnds}, obtaining that
\begin{equation*}
  \begin{aligned}
    \frac{ 1 + E_{\eps,\eta}(t) }{ 1 + \eps E_{\eps,\eta}(t) } \leq \frac{ 1 + E_{\eps,\eta}(0) }{ 1 + \eps E_{\eps,\eta}(0) } e^{C(1-\eps) t} \leq \frac{ 1 + M }{ 1 + \eps M } e^{C t}
  \end{aligned}
\end{equation*}
holds for all $t \in [0 , T_\eta^\eps )$. Consequently, for all $t \in \Big[ 0 , \min \{ T_\eta^\eps , T, \tfrac{1}{C} \ln \big( \tfrac{1+\eps M}{\eps (1 + M)} \big) \} \Big) $ we know that
\begin{equation}\label{Unf-Bd-EE}
  E_{\eps, \eta} (t) \leq \frac{ (1+M) e^{C t} }{1 + \eps M - \eps (1+M) e^{Ct}} \,.
\end{equation}
Notice that the continuity of $E_{\eps, \eta} (t)$ and the maximality of $T_\eta^\eps > 0$ imply that $$ T_\eta^\eps \geq \tfrac{1}{C} \ln \big( \tfrac{1+\eps M}{\eps (1 + M)} \big) > 0. $$ Hence the inequality \eqref{Unf-Bd-EE} holds for all $t \in [ 0, T^\eps )$ uniformly in $\eta > 0$, where $$T^\eps = \min \{ T , \tfrac{1}{C} \ln \big( \tfrac{1 + \eps M}{\eps ( 1 + M )} \big) \} > 0\,.$$ Finally, we can finish the proof of this proposition by standard compactness methods and taking the limit as $\eta \rightarrow 0$. The uniqueness issue can be reduced to the uniqueness of the damped wave map system \eqref{Wave-Map} which can be obtained by methods analogous to those in the book of Shatah and Struwe \cite{Shatah-Struwe(BOOK)-2000}. For convenience, we omit the details of the proof.

\end{proof}

\noindent {\bf The Proof of Theorem \ref{Main-Thm}.} Now, based on the energy estimate \eqref{Bnds-Energ} in Proposition \ref{Prop-WM}, we verify Theorem \ref{Main-Thm}. We observe that the function
$$ f(\eps) : = \frac{1}{C} \ln \Big( \frac{1 + \eps M}{\eps ( 1 + M ) } \Big) $$
is strictly decreasing in $ \eps \in ( 0 , \frac{1}{2} ) $ and $\lim\limits_{\eps \searrow 0} f(\eps) = + \infty$. Consequently, we can choose
$$ \eps_0 = \min \big\{ \tfrac{1}{2} , \tfrac{1}{ ( 1 + M ) e^{CT} - M } \big\} \in (0 ,\tfrac{1}{2}) $$
such that for any $\eps \in ( 0 , \eps_0 )$
$$ f(\eps) = \frac{1}{C} \ln \Big( \frac{1 + \eps M}{\eps ( 1 + M ) } \Big) > T \,. $$
As a result, for the number $T^\eps$ determined in Proposition \ref{Prop-WM}, we have that $T^\eps \equiv T$ for all $\eps \in (0, \eps_0)$. Then, the inequality \eqref{Bnds-Energ} in Proposition \ref{Prop-WM} implies that
\begin{equation*}
  |\partial_t \dr^\eps_R |^2_{L^\infty(0,T;H^2)} + \tfrac{1}{\eps} | \dr^\eps_R |^2_{L^\infty(0,T;H^3)}  \leq \tfrac{ 2 M e^{C T} }{ 1 + \eps_0 M - \eps_0 ( 1 + M ) e^{C T} } : = C_0 < \infty\,,
\end{equation*}
and the proof of Theorem \ref{Main-Thm} is completed.

\qquad\qquad\qquad\qquad\qquad\qquad\qquad \qquad\qquad\qquad\qquad\qquad\qquad\qquad \qquad\qquad\qquad\qquad\qquad\ \ $\Box$

%%%%%%%%%%%%%%%%%%%%%%%%%%%%%%%%%%%%%%%%%%%%%%%%%%%%%%%%%%%%%%%%%%%%%%%%%%%

\section*{Acknowledgement}
The project of this paper was initialized when Ning Jiang visited Arghir Zarnescu at Basque Center for Applied Mathematics (BCAM) in March 2017. They appreciate the hospitality of BCAM. The activity of Ning Jiang on this work was supported by Chinese NSF grant 11471181.

The activity of Arghir Zarnescu on this work was partially supported by a grant of the Ro-manian National Authority for Scientific Research and Innovation, CNCS-UEFISCDI, project number PN-II-RU-TE-2014-4-0657; by the Project of the Spanish Ministry of Economy and Competitiveness with reference MTM2013-40824-P; by the Basque Government through the BERC 2014–2017 program; and by the Spanish Ministry of Economy and Competitiveness MINECO: BCAM Severo Ochoa accreditation SEV-2013-0323.
%%%%%%%%%%%%%%%%%%%%%%%%%%%%%%%%%%%%%%%%%%%%%%%%%%%%%%%%%%%%%%%%%%%%%%%%%%%
\bigskip
% \phantomsection
% \addcontentsline{toc}{section}{\refname}

%\bibliographystyle{unsrtnat}
%\nocite{*}
%\bibliography{reference}

\end{document}